\documentclass{article}%

\usepackage{amsmath}%
\usepackage{amsfonts}%
\usepackage{amssymb}%
\usepackage{graphicx}
\usepackage{a4wide}
\usepackage[utf8]{inputenc}
\usepackage{url}
\usepackage{todonotes}
\usepackage{enumerate}
\usepackage{tikz}
\usepackage{tikz-cd}
\usepackage{hyperref}
\usepackage[all,cmtip]{xy}

\usepackage[style]{abstract}
\usepackage[affil-it]{authblk}

\newtheorem{theorem}{Theorem}[section]

\newtheorem{corollary}[theorem]{Corollary}

\newtheorem{definition}[theorem]{Definition}

\newtheorem{lemma}[theorem]{Lemma}

\newtheorem{proposition}[theorem]{Proposition}

\newcommand{\abGal}[1] {\operatorname{Gal}\big(\overline{#1}/#1\big)}
\newcommand{\PlaceOfK}{\lambda}
\newtheorem{remark}[theorem]{Remark}
\let\oldremark\remark
\renewcommand{\remark}{\oldremark\normalfont}

\numberwithin{equation}{section}

\newenvironment{proof}[1][Proof]{\textbf{#1.} }{\ \rule{0.5em}{0.5em}}

\begin{document}

\title{Explicit surjectivity of Galois representations attached to abelian surfaces and $\operatorname{GL}_2$-varieties}
\date{}

\author{Davide Lombardo
  \thanks{\texttt{davide.lombardo@math.u-psud.fr}}}
\affil{Département de Mathématiques d'Orsay}

\maketitle

\begin{abstract}
Let $A$ be an absolutely simple abelian variety without (potential) complex multiplication, defined over the number field $K$. Suppose that either $\dim A=2$ or $A$ is of $\operatorname{GL}_2$-type: we give an explicit bound $\ell_0(A,K)$ such that, for every prime $\ell>\ell_0(A,K)$, the image of $\abGal{K}$ in $\operatorname{Aut}(T_\ell(A))$ is as large as it is allowed to be by endomorphisms and polarizations.
\end{abstract}


\section{Introduction}
The purpose of this work is the study of Galois representations attached to abelian varieties over number fields, with a particular focus on the case of abelian surfaces. Throughout the paper, the letters $K$ and $A$ will respectively denote a number field and an abelian variety defined over $K$, and the letter $\ell$ will be reserved for prime numbers. 
Along with abelian surfaces, we also consider abelian varieties of $\operatorname{GL}_2$-type, which we define as follows:
\begin{definition}{(cf.~\cite{MR0457455})}\label{surf_def_GL2}
An abelian variety $A/K$ is said to be of $\operatorname{GL}_2$-type if its endomorphism algebra $\operatorname{End}_{\overline{K}}(A) \otimes \mathbb{Q}$ is a totally real number field $E$ such that $[E:\mathbb{Q}]=\dim A$.
\end{definition}

The representations we examine are those given by the natural action of $\abGal{K}$ on the various Tate modules of $A$ (denoted by $T_\ell(A)$), and the problem we study is that of describing the image $G_{\ell^\infty}$ of $\abGal{K}$ in $\operatorname{Aut}\left(T_\ell(A)\right)$.
In a sense to be made precise shortly, we aim to show that this image is as large as it is permitted by some `obvious' constraints, as soon as $\ell$ exceeds a certain bound $\ell_0(A,K)$ that we explicitly compute in terms of arithmetical invariants of $K$ and of the semistable Faltings height of $A$ (which we denote $h(A)$, and for which we use Faltings' original normalization: see \cite[§2.3]{MR3225452}). 

For the abelian varieties we consider -- surfaces and $\operatorname{GL}_2$-type -- this fact, in its qualitative form, has been known since the work of Serre \cite{Serre_resum8586} and Ribet \cite{MR0457455}: the novelty of the result we present here lies in its being completely explicit. Indeed, to the best of the author's knowledge, before the present work the only paper dealing with the problem of effective surjectivity results for abelian surfaces was \cite{MR1998390}, that only covered the case $\operatorname{End}_{\overline{K}}(A)=\mathbb{Z}$ (and was not completely explicit). Unfortunately, the argument of \cite{MR1998390} seems to contain a gap, for in his case analysis the author does not include the subgroup of $\operatorname{GSp}_4(\mathbb{F}_\ell)$ arising from the unique 4-dimensional symplectic representation of $\operatorname{SL}_2$ (case 4 in theorem \ref{surf_thm_Classification}): this is essentially the hardest case, and dealing with it requires nontrivial results of Raynaud on the structure of the action of inertia.

Before stating our main result let us elaborate a little on the `obvious' conditions that are imposed on $G_{\ell^\infty}$. The choice of any $K$-polarization on $A$ equips the Tate modules $T_\ell(A)$ with a bilinear form, the Weil pairing $\langle \cdot, \cdot \rangle$, compatible with the Galois action: this forces $G_{\ell^\infty}$ to be contained in the group of similitudes with respect to the bilinear form $\langle \cdot, \cdot \rangle$. Furthermore, the action of $\abGal{K}$ is also compatible with the natural action of $\operatorname{End}_K(A)$ on $T_\ell (A)$, which in turn implies that $G_{\ell^\infty}$ is contained in the centralizer of $\operatorname{End}_K(A)$ inside $\operatorname{Aut}(T_\ell(A))$.

\medskip

This second condition leads naturally to classifying abelian surfaces according to the structure of $\operatorname{End}_{\overline{K}}(A)$. A study of those rings that appear as endomorphism rings of abelian surfaces (a particular case of the so-called Albert classification, cf. for example \cite[p. 203]{Mumford70abelianvarieties}) leads to the conclusion that only five cases can arise:
\begin{enumerate}
\item Type I, trivial endomorphisms: $A$ is absolutely simple and $\operatorname{End}_{\overline{K}}(A)=\mathbb{Z}$;
\item Type I, real multiplication: $A$ is absolutely simple and $\operatorname{End}_{\overline{K}}(A)$ is an order in a real quadratic field (so $A$ is in particular a $\operatorname{GL}_2$-variety);
\item Type II, quaternionic multiplication: $A$ is absolutely simple and $\operatorname{End}_{\overline{K}}(A)$ is an order in a quaternion division algebra over $\mathbb{Q}$;
\item Type IV, complex multiplication: $A$ is absolutely simple and admits complex multiplication by a quartic CM field;
\item Non-simple case: $A_{\overline{K}}$ is isogenous to the product of two elliptic curves.
\end{enumerate}

\medskip

We focus here on the first three possibilities; the case of complex multiplication (in arbitrary dimension) is treated in \cite{2015arXiv150604734L}, and that of a product of an arbitrary number of elliptic curves without complex multiplication is studied in \cite{2015arXiv150104600L}. It should be possible to combine these results to treat any finite product of elliptic curves (in which some factors admit CM and others do not); this would in particular cover case (5) above.

\subsection{Notation and statement of the result}
We are interested in the Galois representations attached to $A$: the natural action of $\abGal{K}$ on the Tate modules $T_\ell(A)$ gives rise to a family of representations
\[
\rho_{\ell^\infty}: \abGal{K} \to \operatorname{GL}(T_\ell(A))
\]
which will be our main object of study. We will also need to consider the residual mod-$\ell$ representations, which we similarly denote by $\rho_{\ell}: \abGal{K} \to \operatorname{GL}(A[\ell])$; we shall write $G_{\ell^\infty}$ (resp.~$G_\ell$) for the image of $\rho_{\ell^\infty}$ (resp.~$\rho_\ell$). Most of our estimates will be given in terms of the following function:

\begin{definition}\label{surf_def_bFunction}
Let $\alpha(g)=2^{10}g^3$ and define
$
b(d,g,h)=\left( (14g)^{64g^2} d \max\left(h, \log d,1 \right)^2 \right)^{\alpha(g)}.
$
If $K$ is a number field and $A$ is an abelian variety over $K$, we shall use the shorthand $b(A/K)$ to denote $b([K:\mathbb{Q}],\dim A,h(A))$.
\end{definition}


We are now ready to state our main results. 

\begin{theorem}\label{surf_thm_MainZ}
Let $A/K$ be an abelian surface with $\operatorname{End}_{\overline{K}}(A)=\mathbb{Z}$. The equality $G_{\ell^\infty}=\operatorname{GSp}_4(\mathbb{Z}_\ell)$ holds for every prime $\ell$ that satisfies the following three conditions:
\begin{itemize}
\item $\ell > b(2[K:\mathbb{Q}],4,2h(A))^{1/4}$;
\item there exists a place of $K$ of characteristic $\ell$ at which $A$ has semistable reduction;
\item $\ell$ is unramified in $K$.
\end{itemize}
\end{theorem}


For the case of real multiplication we treat the more general situation of abelian varieties of $\operatorname{GL}_2$-type:

\begin{theorem}\label{surf_thm_GL2}
Let $A/K$ be an abelian variety of dimension $g$. Suppose that $\operatorname{End}_{\overline{K}}(A)$ is an order in a totally real field $E$ of degree $g$ over $\mathbb{Q}$ (that is, $A$ is of $\operatorname{GL}_2$-type), and that all the endomorphisms of $A$ are defined over $K$. Let $\ell$ be a prime unramified both in $K$ and in $E$ and strictly larger than 
$
\max\left\{b(A/K)^{g}, \; b(2[K:\mathbb{Q}],2\dim(A),2h(A))^{1/2} \right\}
$: then we have
\[
G_{\ell^\infty}=\left\{x \in \operatorname{GL}_2\left(\mathcal{O}_E \otimes \mathbb{Z}_\ell \right) \bigm\vert \operatorname{det}_{\mathcal{O}_E} x \in \mathbb{Z}_\ell^\times \right\}.
\]
\end{theorem}


The case of abelian surfaces with real multiplication then follows as an immediate consequence:

\begin{corollary}\label{surf_cor_MainRM}
Suppose that $R=\operatorname{End}_{\overline{K}}(A)$ is an order in a real quadratic field $E$ and that all the endomorphisms of $A$ are defined over $K$. Let $\ell$ be a prime unramified both in $K$ and in $E$ and strictly larger than $b(2[K:\mathbb{Q}],4,2h(A))^{1/2}$: then we have
\[
G_{\ell^\infty}=\left\{x \in \operatorname{GL}_2\left(\mathcal{O}_E \otimes \mathbb{Z}_\ell \right) \bigm\vert \operatorname{det}_{\mathcal{O}_E} x \in \mathbb{Z}_\ell^\times \right\}.
\]
\end{corollary}

\begin{remark}\label{surf_rmk_RMStatement}
When $A$ is a surface, the group $H_{\ell^\infty}:=\left\{ x \in \operatorname{GL}_2\left(\mathcal{O}_E \otimes \mathbb{Z}_\ell \right) \bigm\vert \det_{\mathcal{O}_E} x \in \mathbb{Z}_\ell^\times \right\}$ appearing in this statement admits the following concrete description. When $\ell$ is split in $E$, the ring $\mathcal{O}_E \otimes \mathbb{Z}_\ell$ is isomorphic to $\mathbb{Z}_\ell \oplus \mathbb{Z}_\ell$, and we have $H_{\ell^\infty}=
\left\{ (h_1,h_2) \in \operatorname{GL}_2(\mathbb{Z}_\ell)^2 \bigm\vert \det h_1=\det h_2 \right\}$. If, on the other hand, $\ell$ is inert in $E$, then $\mathcal{O}_E \otimes \mathbb{Z}_\ell$ is a domain that contains a canonical copy of $\mathbb{Z}_\ell$ (namely $\mathbb{Z} \otimes \mathbb{Z}_\ell$), and we have $
H_{\ell^\infty}= \left\{ x \in \operatorname{GL}_2(\mathcal{O}_E \otimes \mathbb{Z}_\ell) \bigm\vert \det x \in \mathbb{Z}_\ell^\times \right\}$, where now $\det$ is the usual determinant (since $\mathcal{O}_E \otimes \mathbb{Z}_\ell$ is a domain). More generally, if $A$ is of dimension $g$ and $\ell$ is unramified in $E$, then
$
\displaystyle H_{\ell^\infty} = \left\{ (x_\lambda) \in \prod_{\lambda \mid \ell} \operatorname{GL}_2(\mathcal{O}_\lambda) \bigm\vert \det x_{\lambda_1}=\det x_{\lambda_2} \in \mathbb{Z}_\ell^\times \quad \forall \lambda_1, \lambda_2 \mid \ell \right\},
$
where the product is over the places of $E$ dividing $\ell$.
\end{remark}

Finally, we come to the case of quaternionic multiplication:

\begin{theorem}\label{surf_thm_MainQM}
Let $A/K$ be an abelian surface such that $R=\operatorname{End}_{\overline{K}}(A)$ is an order in an indefinite quaternion division algebra over $\mathbb{Q}$, and let $\Delta$ be the discriminant of $R$. Suppose that all the endomorphisms of $A$ are defined over $K$. If $\ell$ is larger than $b(2[K:\mathbb{Q}],4,2h(A))^{1/2}$, does not divide $\Delta$, and is unramified in $K$, then the equality $G_{\ell^\infty}=\left(R \otimes \mathbb{Z}_\ell\right)^\times$ holds.
\end{theorem}


\begin{remark}
Note that in the case of real and quaternionic multiplication we demand that the endomorphisms of $A$ be defined over $K$, but this is not a severe restriction. Indeed, this condition can be achieved by passing to a finite extension $K'$ of $K$, and when $A$ is a surface it is known that $K'/K$ can be taken to be of degree at most 2 for the case of real multiplication (\cite[Proposition 4.3]{MR1154704}), and at most 12 for the quaternionic case (\cite[Prop. 2.1]{MR2091964}); more complicated (but still explicit) bounds on the degree $K'/K$ are also available in case $A$ is of $\operatorname{GL}_2$-type, cf. again \cite{MR1154704}.

Replacing $K$ with $K'$ corresponds to killing the group of connected components of $G_{\ell^\infty}$, that is, it ensures that the image of Galois is connected. Analogous results (with slightly different bounds) could be stated without this assumption, at the cost of replacing $G_{\ell^\infty}$ by its identity component in the conclusion. 
\end{remark}


Before proceeding with the proof of the three main statements a few more comments are in order. Consider the hypothesis that there is a place of $K$ of characteristic $\ell$ at which $A$ has semistable reduction (this assumption appears in the statement of theorem \ref{surf_thm_MainZ}). Without any assumption on $A$, this condition cannot be turned into an inequality only involving $h(A)$: indeed, the set of primes of unstable reduction of $A/K$ is certainly not invariant under extensions of scalars, so it cannot be controlled just in terms of the \textit{stable} Faltings height; this is really an arithmetical condition that is hard to avoid. On the plus side, the primes which fail to meet this restriction are often easy to determine in practice, especially when $A$ is explicitly given as the Jacobian of a genus 2 curve.

Also note that in many intermediate lemmas we give estimates in terms of the best possible isogeny bound (cf.~section \ref{surf_sect_IsogenyTheorem}), thus avoiding to use the specific form of the function $b(A/K)$. However, in order to make the final results more readable, we have chosen to express them in a form that only involves the function $b$; this also has the merit of giving completely explicit bounds.

\smallskip



Let us also briefly review previous work in the area. As already mentioned, Serre \cite{Serre_resum8586} proved that for a large class of abelian varieties (that includes surfaces with $\operatorname{End}_{\overline{K}}(A)=\mathbb{Z}$) there exists a number $\ell_1(A,K)$ such that $G_{\ell^\infty}=\operatorname{GSp}_{2\dim A}(\mathbb{Z}_\ell)$ for every $\ell$ larger than $\ell_1(A,K)$; his result, however, is not effective, in the sense that the proof does not give any bound on $\ell_1(A,K)$. Similarly, Ribet proved in \cite{MR0457455} an open image result for abelian varieties of $\operatorname{GL}_2$-type that includes surfaces with real multiplication as a particular case, but that is again non-effective. The case of quaternionic multiplication was treated independently in \cite{MR0419368} and in \cite{QMJacobson} by extending the techniques Serre used to prove his celebrated open image theorem for elliptic curves in \cite{MR0387283}, but once again these results were not effective. Dieulefait \cite{MR1969642} considers abelian surfaces $A/\mathbb{Q}$ that satisfy $\operatorname{End}_{\overline{\mathbb{Q}}}(A)=\mathbb{Z}$, and gives sufficient conditions for the equality $G_{\ell^\infty}=\operatorname{GSp}_{4}(\mathbb{Z}_\ell)$ to hold for a fixed prime $\ell$; the form of these conditions, however, is again such that they do not yield a bound for the largest prime for which the equality $G_\ell=\operatorname{GSp}_{4}(\mathbb{Z}_\ell)$ fails to hold. The treatment we give of case 4 of theorem \ref{surf_thm_Classification}, however, has been inspired by Dieulefait's paper.
Finally, notice that \cite[Theorem 5.8]{MR1969642} gives an explicit example of a 2-dimensional Jacobian over $\mathbb{Q}$ for which the equality $G_{\ell^\infty}=\operatorname{GSp}_4(\mathbb{Z}_\ell)$ holds for every prime $\ell \geq 3$ (the result was conditional on Serre's modularity conjecture, which has subsequently been proved by Khare and Wintenberger, cf.~\cite{MR2551763} and \cite{MR2551764}). This result is best possible, in the sense that there exists no principally polarized abelian surface over $\mathbb{Q}$ for which $G_{2^\infty}=\operatorname{GSp}_4(\mathbb{Z}_2)$, cf.~\cite[Proposition 2.5]{2015arXiv150807655Z}.


\smallskip

It is tempting to conjecture that bounds $\ell_0(A,K)$ as above can be taken to be independent of $A$ (albeit obviously not of $K$), but very little is known in this direction: the analogous question for elliptic curves, often called ``Serre's uniformity problem'', is still open even for $K=\mathbb{Q}$. We do know, however, that such bounds must necessarily depend on the dimension of $A$: for example, Theorem 1.4 of \cite{zbMATH06114743} shows that for infinitely many primes $p$ there exists a $\mathbb{Q}$-abelian variety $A_p$ of $\operatorname{GL}_2$-type whose associated $\rho_p$ is not surjective and whose dimension satisfies $\displaystyle \dim A_p \leq \frac{(p-5)(p-7)}{24}$; in particular, a uniform bound for $\ell_0(A,\mathbb{Q})$, if it exists, must grow at least as fast as $\sqrt{\dim A}$.

\medskip

To conclude this introduction let us give a brief overview of the organisation of the paper and of the proof methods. Theorems \ref{surf_thm_MainZ}, \ref{surf_thm_GL2}, and \ref{surf_thm_MainQM} will be shown in sections \ref{surf_sect_EndZ}, \ref{surf_sect_RM}, and \ref{surf_sect_QM} respectively.

The main input for the proof in the case of trivial endomorphism ring comes from group theory, complemented by an application of some nontrivial results of Raynaud. After reducing the problem to that of showing the equality $G_\ell=\operatorname{GSp}_4(\mathbb{F}_\ell)$ for $\ell$ large enough, we recall the classification of the maximal proper subgroups of $\operatorname{GSp}_4(\mathbb{F}_\ell)$ and proceed to show that each of them cannot occur as the image of the Galois representation on $A[\ell]$, at least for $\ell$ large enough. In most cases, this follows from the so-called isogeny theorem of Masser and W\"ustholz \cite{MR1207211} \cite{MR1217345} (theorem \ref{surf_thm_Isogeny} below): in many circumstances, if $G_\ell$ is not maximal, then the Galois module $A[\ell]$ (or $A[\ell] \times A[\ell]$) is nonsimple, a fact that gives rise to minimal isogenies of high degree, eventually contradicting the isogeny theorem for $\ell$ large enough. In some exceptional cases, however, the representations $A[\ell]$ and $A[\ell] \times A[\ell]$ can be irreducible even if $G_\ell$ is comparatively very small, and it is to exclude this possibility that we need to invoke Raynaud's results.
These same results of Raynaud also lead to the following lower bound on the order of $\mathbb{P}G_\ell$ (the projective image of $G_\ell$), which might have some independent interest:
\begin{proposition}{(Proposition \ref{surf_prop_LowerBoundGell})}
Let $A/K$ be any abelian variety of dimension $g$. Let $\ell$ be a prime such that there is a place of $K$ of residual characteristic $\ell$ at which $A$ has either good or bad semistable reduction. If $\ell$ is unramified in $K$ and not less than $g+2$, then $|\mathbb{P}G_\ell| \geq \ell-1$.
\end{proposition}


For the case of real multiplication our method is quite different from the one of \cite{MR0457455}: by appealing to group theory more than it is done in \cite{MR0457455}, we can completely avoid the use of Chebotarev's theorem, which would be the main obstacle in making Ribet's method effective. We show again that if $G_\ell$ is smaller than it could conceivably be (given the restrictions imposed by the Weil pairing and by the action of $\operatorname{End}_{\overline{K}}(A)$), then $A$ admits minimal isogenies of high degree: combined with the isogeny theorem, this gives the desired result.




As for the case of quaternionic multiplication, there is a general philosophy suggesting that -- at the level of Galois representations -- an abelian variety of dimension $2g$ with quaternionic multiplication by an algebra with center $L$ should behave as an abelian variety of dimension $g$ admitting multiplication by $L$, and indeed the case of section \ref{surf_sect_QM} turns out to be the easiest, the argument being very similar to that for elliptic curves without complex multiplication. More precisely, the Tate module decomposes as two copies of a 2-dimensional Galois representation, and we can apply techniques that are an essentially straightforward generalisation of those employed to show analogous results for elliptic curves, and that go back to Serre \cite{MR0387283} (cf. also \cite{MR1209248}).


Finally, in appendix \ref{surf_sect_Index} we show how to bound the index of $\operatorname{End}_{\overline{K}}(A)$ in any order in which it is contained, a result that is needed in the course of the proof of theorems \ref{surf_thm_GL2} and \ref{surf_thm_MainQM}.


\section{Preliminaries}
We collect in this section a number of results that are essentially well-known and that will form the basis for all our further discussion. Specifically, we recall a few fundamental properties of Galois representations attached to abelian varieties and an explicit form (due to Gaudron and Rémond) of the so-called Isogeny Theorem, first proved by Masser and W\"ustholz in a seminal series of papers, cf. especially \cite{MR1207211} and \cite{MR1217345}.

\subsection{Weil pairing, the multiplier of the Galois action}\label{surf_sect_Multiplier}
We fix once and for all a minimal $K$-polarization of $A$. This choice equips the Tate module $T_\ell(A)$ with the Weil pairing, a skew-symmetric, Galois-equivariant form
\[
\langle \cdot, \cdot \rangle : T_\ell(A) \times T_\ell(A) \to \mathbb{Z}_\ell(1),
\]
where $\mathbb{Z}_\ell(1)$ is the 1-dimensional Galois module the action on which is given by the cyclotomic character $\chi_\ell:\abGal{K} \to \mathbb{Z}_\ell^\times$. 
The Weil pairing is known to be nondegenerate on $A[\ell]$ as long as $\ell$ does not divide the degree of the given polarization. We know by \cite[Théorème 1.1]{PolarisationsEtIsogenies} that the degree of a minimal $K$-polarization on $A$ is at most $b(A/K)$: since all the bounds given in the main theorems are strictly larger than this number, we can restrict ourselves to primes that \textit{do not} divide the degree of our fixed polarization, and for which the Weil pairing is therefore nondegenerate. From now on, therefore, we assume that the choice of a minimal polarization has been made:

\smallskip

\noindent\textbf{Convention.} We assume implicitly that a minimal $K$-polarization of $A$ has been chosen. We only consider primes $\ell$ larger than $b(A/K)$, so the Weil pairing induced on $A[\ell]$ by this polarization is nondegenerate.

\medskip

The fact that $\langle \cdot , \cdot \rangle$ is Galois-equivariant means that $G_{\ell^\infty}$ is a subgroup of $\operatorname{GSp}(T_\ell(A), \langle \cdot,\cdot \rangle)$, the group of symplectic similitudes of $T_\ell(A)$ with respect to the Weil pairing, which we will also simply denote $\operatorname{GSp}(T_\ell(A))$. Choosing a $\mathbb{Z}_\ell$-basis of $T_\ell(A)$ we can then consider $G_{\ell^\infty}$ (resp.~$G_\ell$) as being a subgroup of $\operatorname{GSp}_4(\mathbb{Z}_\ell)$ (resp.~$\operatorname{GSp}_4(\mathbb{F}_\ell)$).

\medskip

Our interest in the Weil pairing stems from its relationship with the determinant (or, more precisely, the multiplier) of the Galois action. Let us describe the connection. Recall that if $\langle \cdot,\cdot \rangle$ is a skew-symmetric form, the \textbf{multiplier} of a symplectic similitude $B$ is the only scalar $\nu(B)$ such that $\langle Bv,Bw \rangle=\nu(B) \langle v, w \rangle$ for every $v,w$. The association $B \mapsto \nu(B)$ is then a homomorphism, whose kernel is the group $\operatorname{Sp}(\langle \cdot,\cdot \rangle)$ of symplectic isometries. In the case at hand, the Galois-equivariance of the Weil pairing implies that the composition
\[
\abGal{K} \xrightarrow{\rho_{\ell^\infty}} \operatorname{GSp}\left( T_\ell(A) \right) \xrightarrow{\nu} \mathbb{Z}_\ell^\times
\]
coincides with the cyclotomic character $\chi_\ell:\abGal{K} \to \mathbb{Z}_\ell^\times$, and it follows that 
$G_{\ell^\infty}$ is all of $\operatorname{GSp}(T_\ell(A))$ if and only if it contains $\operatorname{Sp}(T_\ell(A))$ and $\abGal{K} \xrightarrow{\chi_\ell}  \mathbb{Z}_\ell^\times$ is surjective. This latter condition is very easy to check:


\begin{lemma}\label{surf_lemma_Conclusion}
Suppose $\ell$ does not ramify in $K$: then $\abGal{K} \xrightarrow{\chi_\ell} \mathbb{Z}_\ell^\times$ is surjective.
In particular, if $G_\ell$ (resp.~$G_{\ell^\infty}$) contains $\operatorname{Sp}(A[\ell])$ (resp.~$\operatorname{Sp}(T_\ell(A))$) and $\ell$ does not divide the discriminant of $K$, the equality $G_\ell=\operatorname{GSp}(A[\ell])$ (resp. $G_{\ell^\infty}=\operatorname{GSp}(T_\ell(A))$) holds.
\end{lemma}

\begin{proof}
The claim is equivalent to the fact that the equality $[K(\mu_{\ell^n}):K]=\varphi(\ell^n)$ holds for all $n \geq 1$. In particular it suffices to show that $K$ and $\mathbb{Q}(\mu_{\ell^n})$ are linearly disjoint over $\mathbb{Q}$, and since $\mathbb{Q}(\mu_{\ell^n})/\mathbb{Q}$ is Galois it suffices to show that they intersect trivially. Set $L:=K \cap \mathbb{Q}(\mu_{\ell^n})$. Now on one hand $\ell$ is unramified in $K$, so it is \textit{a fortiori} unramified in $L$; on the other hand, every prime different from $\ell$ is unramified in $\mathbb{Q}(\mu_{\ell^n})$, so it is also unramified in $L$: it follows that $L$ is unramified everywhere, that is, $L=\mathbb{Q}$ as claimed.
The second statement is immediate.
\end{proof}

%

\subsection{The isogeny theorem}\label{surf_sect_IsogenyTheorem}
For future reference we introduce here the main tool that will make all the explicit estimates possible. The crucial result is the isogeny theorem of Masser and W\"ustholz \cite{MR1207211} \cite{MR1217345}, in the following completely explicit form proved by Gaudron and Rémond :

\begin{theorem}{(Isogeny Theorem, \cite[Theorem 1.4]{PolarisationsEtIsogenies})}\label{surf_thm_Isogeny}
Let $b(A/K)$ be as in definition \ref{surf_def_bFunction}. For every $K$-abelian variety $A$ and for every $K$-abelian variety $A^*$ that is $K$-isogenous to $A$, there exists a $K$-isogeny $A^* \to A$ whose degree is bounded by $b(A/K)$.
\end{theorem}

It is very likely that the function $b(A/K)$ of definition \ref{surf_def_bFunction} is not the best possible one. Let us then introduce another function $b_0(A/K)$, which is by definition the best possible isogeny bound:
\begin{definition}\label{surf_def_b0Function}
For $A/K$ an abelian variety, let $b_0(A/K)$ be the smallest natural number such that, for every abelian variety $A^*/K$ that is $K$-isogenous to $A$, there exists a $K$-isogeny $A^* \to A$ of degree at most $b_0(A/K)$.
Also let $b_0(A/K;d)=\max_{[K':K] \leq d} b_0(A/K')$, where the maximum is taken over the extensions of $K$ of degree at most $d$.
\end{definition}

It is clear that the isogeny theorem implies that $b_0(A/K)$ and $b_0(A/K;d)$ are finite, and that $b_0(A/K;d) \leq b(d[K:\mathbb{Q}], \dim A, h(A))$. Whenever possible, we will state our results in terms of $b_0$ instead of $b$: in some situations, however, in order to avoid cumbersome expressions involving maxima we simply give bounds in terms of the function $b$.

\subsection{Serre's lifting lemma}
We conclude this section of preliminaries with a well-known lemma of Serre:

\begin{lemma}\label{surf_lemma_LiftingEndZ} Let $n$ be a positive integer, $\ell \geq 5$ be a prime, and $H$ be a closed subgroup of $\operatorname{Sp}_{2n}(\mathbb{Z}_\ell)$ whose projection modulo $\ell$ contains $\operatorname{Sp}_{2n}(\mathbb{F}_\ell)$: then $H=\operatorname{Sp}_{2n}(\mathbb{Z}_\ell)$. Likewise, let $G$ be a closed subgroup of $\operatorname{GSp}_{2n}(\mathbb{Z}_\ell)$ whose projection modulo $\ell$ contains $\operatorname{Sp}_{2n}(\mathbb{F}_\ell)$: then $G'=\operatorname{Sp}_{2n}(\mathbb{Z}_\ell)$.
\end{lemma}
\begin{proof}
The first statement is \cite[Lemme 1 on p. 52]{LettreVigneras}. The second part follows from applying the first to $G=H'$: indeed, the image modulo $\ell$ of $H'$ contains the derived subgroup of $\operatorname{Sp}_{2n}(\mathbb{F}_\ell)$, which (since $\ell \geq 5$) is again $\operatorname{Sp}_{2n}(\mathbb{F}_\ell)$, and the claim follows.
\end{proof}





\newcommand{\TameInertia}{I^t}
\section{Type I -- Trivial endomorphisms}\label{surf_sect_EndZ}
In this section we prove theorem \ref{surf_thm_MainZ}, that is, we establish an explicit surjectivity result under the assumption that $A/K$ is an abelian surface with $\operatorname{End}_{\overline{K}}(A)=\mathbb{Z}$. 
The material is organized as follows. In the first paragraph we recall classical results on the  maximal subgroups of $\operatorname{GSp}_4(\mathbb{F}_\ell)$; notice that -- thanks to lemmas \ref{surf_lemma_Conclusion} and \ref{surf_lemma_LiftingEndZ} -- showing theorem \ref{surf_thm_MainZ} essentially amounts to proving that $G_\ell$ is not contained in any maximal proper subgroup of $\operatorname{GSp}_4(\mathbb{F}_\ell)$. In the second paragraph we collect information about the action of inertia that allows us to conclude that some exceptional subgroups of $\operatorname{GSp}_4(\mathbb{F}_\ell)$ cannot arise as images of Galois representations. Theorem \ref{surf_thm_MainZ} then follows easily, as shown in the last paragraph.

\subsection{Group theory for $\operatorname{GSp}_4(\mathbb{F}_\ell)$}

We start by recalling the classification of the maximal subgroups of $\operatorname{GSp}_4(\mathbb{F}_\ell)$.
The result was first proved by Mitchell in \cite{MR1500986} (see also King's article in \cite{MR2187473} for a more modern account of the result), but we shall instead follow the approach of Aschbacher, who, in \cite{MR746539}, proved a general classification result for the maximal subgroups of the finite classical matrix groups. These maximal subgroups are classified into two categories: ``geometric" subgroups, in turn subdivided into 8 classes $\mathcal{C}_1,\ldots,\mathcal{C}_8$, and ``exceptional" (or ``class $\mathcal{S}$") subgroups.
Among the geometric classes introduced by Aschbacher we will only need to deal with $\mathcal{C}_1, \mathcal{C}_2$ and $\mathcal{C}_3$: 

\begin{definition}
A subgroup $G$ of $\operatorname{GSp}_4(\mathbb{F}_\ell)$ is said to be of class 

\begin{itemize}
\item $\mathcal{C}_1$, if it stabilizes a (totally singular or non-singular) subspace;

\item $\mathcal{C}_2$, if there exist 2-dimensional subspaces $V_1, V_2$ of $\mathbb{F}_\ell^4$ such that $\mathbb{F}_\ell^4 \cong V_1 \oplus V_2$, and
\[
G=\left\{ A \in \operatorname{GSp}_{4}(\mathbb{F}_\ell) \bigm\vert \exists \sigma \in S_{2} \text{ s.t. } AV_i \subseteq V_{\sigma(i)} \text{ for }i=1,2 \right\},
\]
where $S_{2}$ is the group of permutations of the two indices $1,2$.


\newcommand{\structure}{\ast}

\item $\mathcal{C}_3$, if there exists a $\mathbb{F}_{\ell^2}$-structure on $\mathbb{F}_\ell^4$ such that
\[
G=\left\{ A \in \operatorname{GSp}_4(\mathbb{F}_\ell) \bigm\vert \exists \sigma \in \operatorname{Gal}\left(\mathbb{F}_{\ell^2}/\mathbb{F}_\ell \right) : \forall \lambda \in \mathbb{F}_{\ell^2}, \forall v \in \mathbb{F}_\ell^4 \quad A(\lambda \structure v)=\sigma(\lambda) \structure Av  \right\},
\]
where we denote by $\structure$ the multiplication map $\mathbb{F}_{\ell^2} \times \mathbb{F}_\ell^4 \to \mathbb{F}_\ell^4$.
In this case, the set
\[
\left\{ A \in \operatorname{GSp}_4(\mathbb{F}_\ell) \bigm\vert \forall \lambda \in \mathbb{F}_{\ell^2}, \forall v \in \mathbb{F}_\ell^4 \quad A(\lambda \structure v)=\lambda \structure Av \right\}
\]
is a subgroup of $G$ of index $2$ which acts $\mathbb{F}_{\ell^2}$-linearly on $\mathbb{F}_\ell^4$ (for the $\mathbb{F}_{\ell^2}$-structure $\structure$).
\end{itemize}
\end{definition}

For completeness, we also include a precise definition of the exceptional class $\mathcal{S}$ (for more details cf.~\cite[Definition 2.1.3]{MR3098485}):
\begin{definition}\label{def_ClassS}
A subgroup $H$ of $\operatorname{GSp}_{4}(\mathbb{F}_\ell)$ is said to be of class $\mathcal{S}$ if and only if all of the following hold:
\begin{enumerate}
\item $\mathbb{P}H$ is almost simple;
\item $H$ does not contain $\operatorname{Sp}_{4}(\mathbb{F}_\ell)$;
\item $H^{\infty}:=\bigcap_{i \geq 1} H^{(i)}$ acts absolutely irreducibly on $\mathbb{F}_\ell^{4}$.
\end{enumerate}
\end{definition}

A complete list of maximal subgroups of $\operatorname{GSp}_{4}(\mathbb{F}_\ell)$ can be read off \cite[Tables 8.13 and 8.14]{MR3098485}:
\begin{theorem}\label{surf_thm_Classification}
Let $\ell>7$ be a prime number. Let $G$ be a maximal proper subgroup of $\operatorname{GSp}_4(\mathbb{F}_\ell)$ not containing $\operatorname{Sp}_4(\mathbb{F}_\ell)$. One of the following holds:
\begin{enumerate}
\item $G$ is of class $\mathcal{C}_1$;
\item $G$ is of class $\mathcal{C}_2$;
\item $G$ is of class $\mathcal{C}_3$;
\item $G$ is of class $\mathcal{S}$, and is isomorphic to $\operatorname{GL}_2(\mathbb{F}_\ell)$;
\item the projective image $\mathbb{P}G$ of $G$ has order at most 3840.
\end{enumerate}
\end{theorem}

\medskip


Being of class $\mathcal{S}$ in Aschbacher's classification, maximal subgroups of type (4) do not fit in a general geometric picture, but they can still be described in considerable detail. The following lemma follows immediately from the arguments of \cite[§5.3]{MR3098485}:


\begin{lemma}\label{surf_lemma_StructureOfCubicStabilizer}
Let $\ell >7$ be a prime and $G$ be a maximal subgroup of $\operatorname{GSp}_4(\mathbb{F}_\ell)$ of type (4) in the above list. For every $g \in G$, the eigenvalues of $g$ can be written as $\lambda_1^3, \lambda_1^2 \lambda_2, \lambda_1 \lambda_2^2, \lambda_2^3$, where $\lambda_1$ and $\lambda_2$ are the roots of a second-degree polynomial with coefficients in $\mathbb{F}_\ell$.
\end{lemma}
\begin{proof}
In the notation of \cite[§5.3]{MR3098485}, as a representation of $\operatorname{GL}_2(\mathbb{F}_\ell)$ the module $A[\ell]$ is isomorphic to $V_4$, the 3$^{rd}$ exterior power of the standard representation of $\operatorname{GL}_2(\mathbb{F}_\ell)$ (cf.~\cite[Proposition 5.3.6]{MR3098485}). Let
\[
\sigma_4:\operatorname{GL}_2(\mathbb{F}_\ell) \to \operatorname{GSp}_4(\mathbb{F}_\ell) \hookrightarrow \operatorname{GL}(V_4)
\]
be the morphism giving $V_4$ its structure of $\operatorname{GL}_2(\mathbb{F}_\ell)$-module: then for every $\gamma \in \operatorname{GL}_2(\mathbb{F}_\ell)$ the eigenvalues of $\sigma_4(\gamma)$ are $\lambda_1^3,\lambda_1^2\lambda_2,\lambda_1\lambda_2^2,\lambda_2^3$, where $\lambda_1,\lambda_2$ are the eigenvalues of $\gamma$.
\end{proof}
\begin{remark}
The condition $\ell > 7$ ensures the existence of maximal subgroups of this shape, cf.~\cite[Table 8.13]{MR3098485}.
\end{remark}

\begin{remark}\label{surf_rmk_StabCubic} Maximal subgroups of type (4) can be interpreted in terms of morphisms of algebraic groups: the unique irreducible 4-dimensional representation of the algebraic group $\operatorname{SL}_2$ is symplectic, so it gives rise to an embedding $\operatorname{SL}_2 \hookrightarrow \operatorname{Sp}_4$ with weights $-3,-1,1,3$. This representation extends to a representation $\operatorname{GL}_2 \to \operatorname{GSp}_4$, and subgroups of type (4) correspond to the $\mathbb{F}_\ell$-points of the image of this representation.
\end{remark}

\subsection{The action of inertia}
For the whole section $\PlaceOfK$ is a fixed place of $\mathcal{O}_K$ at which $A$ has \textit{semistable} reduction. We write $\ell$ for the rational prime below $\PlaceOfK$ and $e$ for the absolute ramification index of $\PlaceOfK$. We also let $I$ be the inertia group at $\PlaceOfK$ and $\TameInertia$ be the corresponding \textit{tame} inertia group, that is, the quotient of $I$ by its maximal pro-$\ell$ subgroup.
The action of $I$ on $A[\ell]$ is described by the following celebrated result of Raynaud:

\begin{theorem}{(\cite[Corollaire 3.4.4]{MR0419467})}\label{surf_thm_RaynaudpSchemas}
Let $V$ be a simple Jordan-H\"older quotient of $A[\ell]$ (as a module over $I$). Suppose that $V$ has dimension $n$ over $\mathbb{F}_\ell$. The action of $I$ on $A[\ell]$ factors through $\TameInertia$. Moreover, there exist integers $e_1, \ldots, e_n$ such that:
\begin{itemize}
\item $V$ has a structure of $\mathbb{F}_{\ell^n}$-vector space
\item the action of $\TameInertia$ on $V$ is given by a character $\psi:\TameInertia \to \mathbb{F}_{\ell^n}^\times$
\item $\psi=\varphi_1^{e_1} \ldots \varphi_n^{e_n}$, where $\varphi_1, \ldots, \varphi_n$ are the fundamental characters of $\TameInertia$ of level $n$
\item for every $i=1, \ldots, n$ the inequality $0 \leq e_i \leq e$ holds
\end{itemize}
\end{theorem}

\begin{remark}\label{threefolds_rmk_SemistableReduction}
Raynaud's theorem is usually stated for places of \textit{good} reduction. However, as it was shown in \cite[Lemma 4.9]{MR3211798}, the extension to the semistable case follows easily upon applying results of Grothendieck \cite{SGAMonodromie}.
\end{remark}

In this section we only consider $A[\ell]$ as a $\TameInertia$-module; in particular, by a Jordan-H\"older quotient of $A[\ell]$ we implicitly mean ``under the action of $\TameInertia$". Notice that, in view of Raynaud's theorem, the map $\rho_\ell : \abGal{K} \to \operatorname{Aut} A[\ell]$ induces a map $\TameInertia \to \operatorname{Aut} A[\ell]$, which for simplicity of notation we still denote $\rho_\ell$.
Recall furthermore that the norm, taken from $\mathbb{F}_{\ell^n}$ to $\mathbb{F}_{\ell}$, of a fundamental character of level $n$ is the unique fundamental character of level 1; when $\ell$ is unramified in $K$, this fundamental character of level 1 is $\chi_\ell$, the cyclotomic character mod $\ell$.

\medskip

Recall the following fact (a consequence of the definition of the fundamental characters):
\begin{lemma}\label{surf_lemma_HigherCharacters}
If $\varphi : \TameInertia \to \mathbb{F}_{\ell^n}^\times$ is any fundamental character of level $n$, then $\varphi$ is surjective, hence in particular its image is a cyclic group of order $\ell^n-1$.
\end{lemma}

\begin{lemma}\label{surf_lemma_dimJHQuotients}
Let $W$ be a simple Jordan-H\"older quotient of $A[\ell]$ of dimension $n$ and let $\psi$ be the associated character $\TameInertia \to \mathbb{F}_{\ell^n}^\times$. The image of $\psi$ is not contained in $\mathbb{F}_{\ell^k}^\times$ for any $k\bigm|n$, $k<n$.
\end{lemma}
\begin{proof}
Suppose by contradiction that the image of $\psi$ is contained in $\mathbb{F}_{\ell^k}^\times$ for a certain $k \geq 1$, and let $\sigma$ be a generator of $\operatorname{Gal}\left(\mathbb{F}_{\ell^k}/\mathbb{F}_\ell \right)$. Since the action of $\TameInertia$ on $W$ can be diagonalized over $\mathbb{F}_{\ell^k}$, we can find a vector $v \in W \otimes_{\mathbb{F}_\ell} \mathbb{F}_{\ell^k}$ that is a common eigenvector for the action of $\TameInertia$. The $\mathbb{F}_{\ell^k}$-vector subspace of $W \otimes_{\mathbb{F}_\ell} \mathbb{F}_{\ell^k}$ spanned by $v, \sigma v, \ldots, \sigma^{k-1} v$ is by construction $\sigma$-stable, hence it descends to a $\mathbb{F}_\ell$-subspace $W'$ of $W$, and it is clear by construction that $W'$ is also stable under the action of $\TameInertia$. As $W$ is irreducible and $W'$ is nontrivial we must have $W'=W$, and since $\dim W' \leq k$ we have $n=\dim W \leq k$, contradiction.
\end{proof}

\begin{lemma}\label{surf_lemma_dim2}
Let $\ell>7$, and suppose $G_\ell$ is contained in a maximal subgroup of type (4) in the sense of theorem \ref{surf_thm_Classification}. Then every simple Jordan-H\"older quotient $W$ of $A[\ell]$ has dimension at most 2.
\end{lemma}
\begin{proof}
The hypothesis and lemma \ref{surf_lemma_StructureOfCubicStabilizer} imply that all the eigenvalues of every element of $G_\ell$ lie in $\mathbb{F}_{\ell^2}$, hence in particular the same is true for the eigenvalues of the action of $\TameInertia$. It follows that the image of the character $\psi$ associated with $W$ is entirely contained in $\mathbb{F}_{\ell^2}$, and lemma \ref{surf_lemma_dimJHQuotients} shows that $W$ is of dimension at most 2.
\end{proof}

\begin{corollary}
Suppose $\ell>7$ is unramified in $K$ and $G_\ell$ is contained in a maximal subgroup of type (4) in the sense of theorem \ref{surf_thm_Classification}. Let $W$ be an $n$-dimensional Jordan-H\"older quotient of $A[\ell]$. The following are the only possibilities:
\begin{itemize}
\item $n=1$ and $\TameInertia$ acts trivially on $W$;
\item $n=1$ and $\TameInertia$ acts on $W$ through $\chi_\ell$;
\item $n=2$ and $\TameInertia$ acts on $W$ through a fundamental character of level $2$.
\end{itemize}
\end{corollary}
\begin{proof}
Lemma \ref{surf_lemma_dim2} implies $n \leq 2$, and theorem \ref{surf_thm_RaynaudpSchemas} classifies all simple Jordan-H\"older quotients of dimension at most 2 (notice that $\ell$ is unramified, so $e=1$).
\end{proof}

\smallskip

From now on we suppose that $\ell>7$ is unramified in $K$ and that $G_\ell$ is contained in a maximal subgroup of type (4) in the sense of theorem \ref{surf_thm_Classification}. We fix a Jordan-H\"older filtration of the $\TameInertia$-module $A[\ell]$, and we denote by $W_1,\ldots,W_k$ the simple subquotients of the filtration.

Furthermore, we let $m_0, m_1, m_2$ be the number of $W_i$ on which the action of $\TameInertia$ is respectively trivial, given by $\chi_\ell$, or given by a fundamental character of level 2. Clearly, by comparing dimensions, we must have
\begin{equation}\label{surf_eq_det}
m_0+m_1+2m_2=\dim_{\mathbb{F}_\ell} A[\ell]=4.
\end{equation}
Considering the determinant of the Galois action gives a second numerical relation:
\begin{lemma}
We have $m_1+m_2=2$.
\end{lemma}
\begin{proof}
Let $W$ be a simple Jordan-H\"older subquotient $W$ of $A[\ell]$. The determinant of the action of $g \in \TameInertia$ on $W$ is either $1$ (if the action is trivial) or $\chi_\ell(g)$ (if the action is given by $\chi_\ell$ or by a fundamental character of level 2).
On the other hand, the determinant of the Galois action of $g \in \TameInertia$ on $A[\ell]$ is $\chi_\ell(g)^2$ by the properties of the Weil pairing, so we must have
\[
\prod_{W_i} \det\left(\rho_\ell(g):W_i \to W_i \right) = \chi_\ell(g)^2 \quad \forall g \in \TameInertia,
\]
or in another words $\chi_\ell^{m_1+m_2-2}(g)=1$ for all $g \in \TameInertia$. Since $\chi_\ell(\TameInertia)$ is a cyclic group of order $\ell-1$, this implies $\ell-1 \bigm\vert m_1+m_2-2$, and since $|m_1+m_2-2| \leq 2 < \ell-1$ we must in fact have $m_1+m_2-2=0$, as claimed.
\end{proof}

\begin{corollary}\label{surf_cor_EigenvaluesMultisets}
Suppose $\ell>7$ is unramified in $K$ and $G_\ell$ is contained in a maximal subgroup of type (4) in the sense of theorem \ref{surf_thm_Classification}.
For every $g \in \TameInertia$ the multiset of eigenvalues of $\rho_\ell(g)$ is one of the following, where we denote by $\varphi_1$ and $\varphi_2=\varphi_1^\ell$ the two fundamental characters of level 2:
\begin{enumerate}
\item $\left\{\chi_\ell(g),\chi_\ell(g),1,1\right\}$ 
\item $\left\{ \varphi_1(g), \varphi_2(g), \chi_\ell(g), 1 \right\}$ 
\item $\left\{ \varphi_1(g), \varphi_2(g), \varphi_1(g), \varphi_2(g) \right\}$ 
\end{enumerate}
\end{corollary}

\begin{proof}
The multiset of eigenvalues of $\rho_\ell(g)$ is the union of the multisets of values taken by the characters that give the action of $g$ on the simple Jordan-H\"older factors $W_i$. In turn, these characters and their multiplicities are determined by the three numbers $m_0,m_1,m_2$, which by equation \eqref{surf_eq_det} and the previous lemma satisfy $m_1=m_0$ and $m_2=2-m_0$. The three cases in the statement correspond to the three possibilities $m_0=0$, $m_0=1$, and $m_0=2$.
\end{proof}

\smallskip

We can now show that groups of type (4) do not arise as images of Galois representations (at least when $\ell>7$ is unramified in $K$): 
\begin{proposition}\label{surf_prop_NoTwistedCubic}
Let $\ell>7$ be a prime unramified in $K$. Suppose there is a place $\PlaceOfK$ of $K$ of characteristic $\ell$ at which $A$ has semistable reduction: then $G_\ell$ is not contained in a group of type (4) in the sense of theorem \ref{surf_thm_Classification}.
\end{proposition}

\begin{proof}
Suppose on the contrary that $G_\ell$ is contained in a group of type (4). By lemma \ref{surf_lemma_StructureOfCubicStabilizer}, for any $g \in \abGal{K}$ the eigenvalues of $\rho_\ell(g)$ are of the form $\left\{ a^3, a^2d, ad^2, d^3\right\}$ for certain $a,d \in \mathbb{F}_{\ell^2}^\times$. This applies in particular to the tame inertia group $\TameInertia$: for every $g$ in $\TameInertia$, the eigenvalues of $\rho_\ell(g)$ lie in $\mathbb{F}_{\ell^2}^\times$, and -- taken in some order $\lambda_1, \lambda_2,\lambda_3,\lambda_4$ -- they satisfy the system of equations
\begin{equation}\label{surf_eq_EigenvaluesSL}
\lambda_1\lambda_4=\lambda_2\lambda_3, \; \lambda_2\lambda_4=\lambda_3^2, \; \lambda_1\lambda_3=\lambda_2^2.
\end{equation}

We now go through the cases listed in corollary \ref{surf_cor_EigenvaluesMultisets} and see that (for a suitably chosen $g \in \TameInertia$) there is no way to renumber the multiset of eigenvalues of $\rho_\ell(g)$ in such a way that the three equations above are all satisfied together, a contradiction that shows the result. Consider case 1. By lemma \ref{surf_lemma_HigherCharacters}, the condition $\ell > 7$ implies that the order of $\chi_\ell(\TameInertia)$ is at least 10, so there exists a $g \in \TameInertia$ with $\chi_\ell(g) \neq \pm 1$. Consider equations \eqref{surf_eq_EigenvaluesSL} for this specific $g$. If either $\lambda_2$ or $\lambda_3$ is $1$, then one of the last two equations reads $\chi_\ell(g)^d=1$ with $d=1$ or $2$, which contradicts $\chi_\ell(g) \neq 1,-1$. But if neither $\lambda_2$ nor $\lambda_3$ is $1$ then the only possibility is $\lambda_1=\lambda_4=1$, $\lambda_2=\lambda_3=\chi_\ell(g)$, for which none of the three equations is satisfied.
Likewise, in case 3 we can choose a $g \in \TameInertia$ such that $\varphi_1^{(2)}(g)$ is of order $\ell^2-1$, and -- independently of the numbering of the eigenvalues -- from equations \eqref{surf_eq_EigenvaluesSL} we obtain $\left(\varphi_1^{(2)}(g)\right)^2=\left(\varphi_2^{(2)}(g)\right)^2$, which using $\varphi_2^{(2)}=\left(\varphi_1^{(2)}\right)^\ell$ implies $\left(\varphi_1^{(2)}(g)\right)^{2(\ell-1)} = 1$, a contradiction. 
Finally, if we are in case 2, then taking the norm from $\mathbb{F}_{\ell^2}$ to $\mathbb{F}_{\ell}$ of the three equations \eqref{surf_eq_EigenvaluesSL} we find that for all $g \in \TameInertia$ there exists a positive integer $d \leq 3$ such that $\chi_\ell(g)^d = 1$, which again contradicts the fact that $\chi_\ell(\TameInertia)$ has order at least $10$.
\end{proof}

\medskip

Considering the action of inertia also allows us to exclude groups of type (5) in the list of theorem \ref{surf_thm_Classification}. The following proposition (which applies to abelian varieties of arbitrary dimension and with arbitrary endomorphism ring) gives a linear lower bound on the order of $\mathbb{P}G_\ell$:

\begin{proposition}\label{surf_prop_LowerBoundGell}
Let $A/K$ be any abelian variety of dimension $g$. Let $\ell$ be a prime such that there is a place $\PlaceOfK$ of $K$ of residual characteristic $\ell$ at which $A$ has either good or bad semistable reduction. If $\ell$ is unramified in $K$ and not less than $g+2$, then $|\mathbb{P}G_\ell| \geq \ell-1$.
\end{proposition}
\begin{proof}
We let again $W_1, \ldots, W_k$ be the simple Jordan-H\"older subquotients of $A[\ell]$ under the action of the tame inertia group at $\PlaceOfK$ (denote $\TameInertia$), and we write $\psi_i$ ($i=1,\ldots,k$) for the character giving the action of $\TameInertia$ on $W_i$. Let $N$ be the order of $|\mathbb{P}G_\ell|$, and notice that for every $y \in G_\ell$ the projective image of $y^{N}$ is trivial, that is, $y^{N}$ is a multiple of the identity, and in particular has a unique eigenvalue of multiplicity $2g$. Since for $x \in \TameInertia$ the eigenvalues of $\rho_\ell(x)$ are given by the Galois conjugates of the various $\psi_i(x)$, this implies that for all $i,j =1,\ldots,k$, for all integers $t \geq 0$, and for all $x \in \TameInertia$ we have
\begin{equation}\label{threefolds_eq_ConstantGroupsEigenvalues}
\psi_i(x)^{\ell^t N}=\psi_j(x)^{N}.
\end{equation}
We now distinguish three cases:
\begin{enumerate}
\item At least one of the $W_i$'s is of dimension $\geq 2$: without loss of generality, we can assume that $n:=\dim W_1$ is at least 2. Let $\psi$ be the associated character. By Raynaud's theorem, there are integers $e_0,\ldots,e_{n-1} \in \{0,1\}$ such that $\psi=\varphi^{\sum_{i=0}^{n-1} e_i \ell^i}$, where $\varphi$ is a fundamental character of level $n$. Note that we cannot have $e_i=1$ for all $i=0,\ldots,n-1$, for otherwise we would have $\psi=\chi_\ell$, which contradicts the fact that $W_1$ is of dimension $n > 1$ (lemma \ref{surf_lemma_dimJHQuotients}). In particular, since for all integers $t \geq 0$ the character $\varphi^{\ell^t}$ is a Galois conjugate of $\varphi$, replacing $\varphi$ with $\varphi^{\ell^t}$ for a suitable $t$ we can assume that $e_{n-1}=0$ (notice that replacing $\varphi$ with $\varphi^\ell$ has the effect of permuting cyclically the integers $e_i$, at least one of which is zero). Now $\varphi$ has exact order $\ell^{n}-1$, so $\psi=\varphi^{\sum_{i=0}^{n-1} e_i \ell^i}$ has order at least
\[
\frac{\ell^n-1}{\sum_{i=0}^{n-1} e_i \ell^i} \geq \frac{\ell^n-1}{\sum_{i=0}^{n-2} \ell^i}= \frac{(\ell^n-1)(\ell-1)}{(\ell^{n-1}-1)} \geq \ell(\ell-1),
\]
that is to say, there is an $x \in \TameInertia$ such that $\psi(x)$ has order at least $\ell(\ell-1)$.
Consider now equation \eqref{threefolds_eq_ConstantGroupsEigenvalues} for this specific $x$, for $\psi_i=\psi_j=\psi$ and for $t=1$: it gives $\psi(x)^{(\ell-1) \cdot N}=1$, so $\psi(x)$ has order at most $(\ell-1) \cdot N$. Thus we obtain $(\ell-1) \cdot N \geq \ell(\ell-1)$, that is $N \geq \ell > \ell-1$ as claimed.

\item All the $W_i$'s are of dimension 1, for at least one index $i$ we have $\psi_i=1$, and for at least one index $j$ we have $\psi_j=\chi_\ell$: then for all $x \in \TameInertia$ we have $\psi_j(x)^{N}=\psi_i(x)^{N}$, that is, $\chi_\ell(x)^{N}=1$ for all $x \in \TameInertia$. As $\chi_\ell$ has exact order $\ell-1$, this implies $N \geq \ell-1$. 
\item All the $W_i$'s are of dimension 1 and all characters $\psi_i$ are equal to each other (and in particular are either all trivial or all equal to the cyclotomic character $\chi_\ell$): in this case there are exactly $k=2g$ simple Jordan-H\"older quotients, and from the equality
\[
\chi_\ell(x)^{g} = \det \rho_\ell(x)=\prod_{i=1}^{2g} \psi_i(x) = \begin{cases} 1, \text{ if }\psi_i=1\text{ for every }i \\ \chi_\ell(x)^{2g}, \text{ if }\psi_i=\chi_\ell\text{ for every }i \end{cases}
\]
we find $\chi_\ell(x)^g=1$ for all $x \in \TameInertia$, which contradicts the fact that the order of $\chi_\ell$ is $\ell-1 > g$.
\end{enumerate}
\end{proof}

\subsection{The surjectivity result}

We are almost ready to prove theorem \ref{surf_thm_MainZ}: the last ingredient we are missing are two isogeny estimates which form the subject of lemmas \ref{surf_lemma_Irreducibility} and \ref{surf_lemma_AbsoluteIrreducibility} below. The strategy of proof of these lemmas is a variant of the approach of \cite{MR1209248} -- cf.~especially lemmas 3.1 and 3.2 of \textit{op.~cit.}

\begin{lemma}\label{surf_lemma_Irreducibility}
Let $A/K$ be an abelian variety of dimension $g$ with $\operatorname{End}_{K}(A) = \mathbb{Z}$, and let $\ell$ be a prime strictly larger than $b_0(A/K)$. The $G_\ell$-module $A[\ell]$ is irreducible.
\end{lemma}

\begin{proof}
Suppose by contradiction that $A[\ell]$ is not irreducible and let $H$ be a nontrivial subspace of $A[\ell]$ stable under the action of $G_\ell$. As $H$ is a proper $\mathbb{F}_\ell$-subspace of $A[\ell] \cong \mathbb{F}_\ell^{2g}$, its order divides $\ell^{2g-1}$.
Consider now the abelian variety $A^*=A/H$, which is defined over $K$ (since $H$ is), and let $\Psi:A^* \to A$ be an isogeny of degree at most $b_0(A/K)$. Let $\pi: A \to A^*$ be the canonical projection, of degree $|H| \leq \ell^{2g-1}$, and consider the composition $\Psi \circ \pi:A \to A$. By the hypothesis $\operatorname{End}_K(A)=\mathbb{Z}$ this composition must be multiplication by $m$ for a certain nonzero integer $m$. Comparing degrees we see that $m^{2g} = \deg(\Psi) \deg(\pi) \leq b_0(A/K) \ell^{2g-1}$, and on the other hand $\Psi \circ \pi$ kills $H$ (since this is true even for $\pi$ alone), so $mH=0$. Every nonzero element of $H$ has order $\ell$, so $m$ must be divisible by $\ell$, which implies $\ell^{2g} \leq b_0(A/K) \ell^{2g-1}$, i.e. $\ell \leq b_0(A/K)$, contradiction.
\end{proof}

\begin{lemma}\label{surf_lemma_AbsoluteIrreducibility}
Let $A/K$ be an abelian variety of dimension $g$. Suppose that $\operatorname{End}_{K}(A) = \mathbb{Z}$ and let $\ell$ be a prime strictly larger than $b_0(A^2/K)^{1/2g}$: the centralizer of $G_\ell$ inside $\operatorname{End}(A[\ell])$ is $\mathbb{F}_\ell$.
\end{lemma}

\begin{proof}
Suppose by contradiction that the centralizer of $G_\ell$ inside $\operatorname{End}(A[\ell])$ is strictly larger than $\mathbb{F}_\ell$ and choose an $\alpha$ lying in this centralizer but not in $\mathbb{F}_\ell$. Consider the abelian variety $B=A^2$ and the subgroup of $B$ given by $\Gamma=\left\{ (x,\alpha x) \bigm\vert x \in A[\ell]\right\}$. Note that $\Gamma$ is defined over $K$: indeed any $g \in \abGal{K}$ sends $(x,\alpha x)$ to $(\rho_\ell(g) x, \rho_\ell(g) \alpha x) =\left(\rho_\ell(g) x, \alpha (\rho_\ell(g) x) \right) \in \Gamma$ (since $\alpha$ commutes with all of $\rho_\ell(\abGal{K})$).
Let $B^*=B/\Gamma$, $\pi:B \to B^*$ be the natural projection, and $\psi:B^* \to B$ be an isogeny in the opposite direction satisfying $\deg(\psi) \leq b_0(A^2/K)$.

The endomorphism $\psi \circ \pi$ of $B$ kills $\Gamma$, and thanks to the hypothesis $\operatorname{End}_{K}(A)=\mathbb{Z}$ it can be represented by a matrix $\left( \begin{matrix} a & b \\ c & d \end{matrix} \right) \in \operatorname{M}_2(\operatorname{End}_{K}(A))=\operatorname{M}_2(\mathbb{Z})$ with nonzero determinant. By definition we must have $ax+b\alpha x=cx+d\alpha x=0$ for every $x \in A[\ell]$. Suppose that one among $a,b,c,d$ is not divisible by $\ell$ (and for the sake of simplicity let us assume it is $b$, the other cases being analogous): then we have $\alpha(x)=-b^{-1}a x$ for every $x$ in $A[\ell]$, which shows that $\alpha$ coincides with the multiplication by an element of $\mathbb{F}_\ell$, contradiction. Therefore $a,b,c,d$ are all divisible by $\ell$, and the degree of $\psi \circ \pi$, which is $\left(\det\left( \begin{matrix} a & b \\ c & d \end{matrix} \right)\right)^{2g}$, is divisible by $\ell^{4g}$. On the other hand, the degree of $\pi$ is $|\Gamma|=|A[\ell]|=\ell^{2g}$, so we deduce 
\[
\ell^{4g} \leq \operatorname{deg}(\psi \circ \pi) = \deg \psi \cdot \ell^{2g} \leq \ell^{2g} b_0(A^2/K),
\]
which is against the hypothesis.
\end{proof}

\begin{theorem}{(Theorem \ref{surf_thm_MainZ})} 
Let $A/K$ be an abelian surface with $\operatorname{End}_{\overline{K}}(A)=\mathbb{Z}$. The equality $G_{\ell^\infty}=\operatorname{GSp}_4(\mathbb{Z}_\ell)$ holds for every prime $\ell$ that satisfies the following three conditions:
\begin{itemize}
\item $\ell > b(2[K:\mathbb{Q}],4,2h(A))^{1/4}$;
\item there exists a place of $K$ of characteristic $\ell$ at which $A$ has semistable reduction;
\item $\ell$ is unramified in $K$.
\end{itemize}
\end{theorem}

\begin{proof}
We claim that $G_{\ell}$ contains $\operatorname{Sp}_4(\mathbb{F}_\ell)$. If this is not the case, 
then $G_{\ell}$ is contained in one of the maximal proper subgroups of $\operatorname{GSp}_4(\mathbb{F}_\ell)$ listed in theorem \ref{surf_thm_Classification}. Let us go through this list.

Case (1) is excluded by lemma \ref{surf_lemma_Irreducibility}, since $\ell > b(2[K:\mathbb{Q}],4,2h(A))^{1/4} > b_0(A/K)$.
Suppose next that $G$ is contained in a maximal subgroup of type (2): then there exists a (normal) subgroup $H_\ell$ of $G_\ell$, of index at most 2, that does not act irreducibly on $A[\ell]$. Consider the morphism \[\abGal{K} \to G_\ell \to G_\ell/H_\ell:\] the fixed field of its kernel is a certain extension $K'$ of $K$, of degree at most 2, such that $\rho_\ell\left(\abGal{K'} \right)=H_\ell$ does not act irreducibly on $A[\ell]$. Applying lemma \ref{surf_lemma_Irreducibility} to $A/K'$ we find $\ell \leq b_0(A/K') \leq b(2[K:\mathbb{Q}],4,2h(A))^{1/4}$, contradiction.

Consider now case (3): there is a (normal) subgroup $H_\ell$ of $G_\ell$, of index at most 2, and a structure of $\mathbb{F}_{\ell^2}$-vector space on $A[\ell]$, such that the action of $H_\ell$ on $A[\ell]$ is compatible with the $\mathbb{F}_{\ell^2}$-structure. As above, let $K''$ be the fixed field of $\ker\left(\abGal{K} \to G_\ell \to G_\ell/H_\ell\right)$, and notice that $K''$ is a (at most) quadratic extension of $K$ such that the centralizer of $\rho_\ell(\abGal{K''})$ contains a copy of $\mathbb{F}_{\ell^2}$: this contradicts lemma \ref{surf_lemma_AbsoluteIrreducibility}, since $\ell > b(2[K:\mathbb{Q}],4,2h(A))^{1/4} \geq b_0(A^2/K')^{1/4}$, so we cannot be in case (3) either.
Finally, case (4) is impossible by proposition \ref{surf_prop_NoTwistedCubic}, and case (5) is excluded by proposition \ref{surf_prop_LowerBoundGell}: indeed we clearly have $\ell > 3842$, so proposition \ref{surf_prop_LowerBoundGell} implies that $|\mathbb{P}G_\ell| > 3840$. 
It follows as claimed that $G_\ell$ contains $\operatorname{Sp}_4(\mathbb{F}_\ell)$, hence $G_{\ell^\infty}$ contains $\operatorname{Sp}_4(\mathbb{Z}_\ell)$ by lemma \ref{surf_lemma_LiftingEndZ}; lemma \ref{surf_lemma_Conclusion} finally implies the desired equality $G_{\ell^\infty}=\operatorname{GSp}_4(\mathbb{Z}_\ell)$.
\end{proof}

\section{Type I -- Real multiplication}\label{surf_sect_RM}
We consider now the case of $\operatorname{GL}_2$-varieties, which includes abelian surfaces with real multiplication as a special case. Recall (definition \ref{surf_def_GL2}) that an abelian variety $A$ is said to be of $\operatorname{GL}_2$-type when it is absolutely simple and $R=\operatorname{End}_{\overline{K}}(A)$ is an order in a totally real number field $E$ of degree equal to $\dim A$. We shall assume that the action of $R$ is defined over $K$.
For every $\ell$ we set $\mathcal{O}_\ell=\mathcal{O}_E \otimes_{\mathbb{Z}} \mathbb{Z}_\ell$, and if $\lambda$ is a place of $E$ we let $\mathcal{O}_\lambda$ be the completion of $\mathcal{O}_E$ at $\lambda$. We have $\mathcal{O}_\ell \cong \prod_{\lambda | \ell} \mathcal{O}_\lambda$, where the product is over the places of $E$ dividing $\ell$. An implicit convention will always be in force, that if $\lambda$ is a place of $E$ then $\ell$ denotes its residual characteristic.

\begin{definition}
Following Ribet's paper \cite{MR0457455} we say that a rational prime $\ell$ is \textbf{good} for $A$ if it does not divide the index $[\mathcal{O}_E:R]$.
\end{definition}

As $[\mathcal{O}_E:R]$ is finite, all but finitely many primes are good for $A$. It is a general fact that $[\mathcal{O}_E:R]$ can be bounded in terms of $K$ and $h(A)$, cf.~appendix \ref{surf_sect_Index}. We obtain from proposition \ref{surf_prop_EndomorphismRingIndex} the following estimate, which enables us to assume that all the primes we work with are good:

\begin{proposition}\label{surf_prop_EveryPrimeIsGood}
The index $[\mathcal{O}_E:R]$ is bounded by $b\left ([K:\mathbb{Q}],\dim A,h(A) \right)^{\dim A}$. In particular, any prime strictly larger than this quantity is good.
\end{proposition}

\smallskip

\noindent From now on we only consider \textit{good} primes that do not ramify in $E$ -- this only excludes finitely many primes.
Notice that in the case of surfaces, in view of the last proposition and of the obvious inequality $b(2[K:\mathbb{Q}],2\dim A,2h(A))^{1/2} > b\left ([K:\mathbb{Q}],\dim A,h(A) \right)^2$, all the primes considered in corollary \ref{surf_cor_MainRM} are good for $A$. For any good prime $\ell$ we have $R_\ell:=R \otimes \mathbb{Z}_\ell \cong \mathcal{O}_E \otimes \mathbb{Z}_\ell$, and furthermore for such primes the structure of $T_\ell(A)$ is well understood:

\begin{proposition}{(\cite[Proposition 2.2.1]{MR0457455})}
If $\ell$ is good for $A$, then $T_\ell(A)$ is a free $R_\ell$-module of rank 2; equivalently, it is a free $\mathcal{O}_\ell$-module of rank 2.
\end{proposition}


When $\ell$ is good and $\lambda$ is a place of $E$ of characteristic $\ell$ we put $T_\lambda(A)=T_\ell(A) \otimes_{\mathcal{O}_\ell} \mathcal{O}_\lambda$: this makes sense since $\mathcal{O}_\ell=R \otimes \mathbb{Z}_\ell$. The Galois action on $T_\ell(A)$ is $\mathcal{O}_\ell$-linear, and we thus obtain canonical decompositions $\displaystyle T_\ell(A) \cong \prod_{\lambda|\ell} T_\lambda(A)$; the $\mathcal{O}_\ell$-linear morphism $\rho_{\ell^\infty}$ then amounts to a family of $\mathcal{O}_\lambda$-linear maps
\[
\rho_{\lambda^\infty}:\abGal{K} \to \operatorname{GL}(T_\lambda(A)) \cong \operatorname{GL}_2 \left(\mathcal{O}_{\lambda}\right).
\]

We also have isomorphisms $\operatorname{Aut} T_\ell(A) \cong \operatorname{GL}_2(\mathcal{O}_E \otimes \mathbb{Z}_\ell) \cong \prod_{\lambda|\ell} \operatorname{GL}_2 \left(\mathcal{O}_{\lambda}\right)$, and we regard the $\ell$-adic Galois representation on $T_\ell(A)$ as a group morphism
\[
\rho_{\ell^\infty}:\abGal{K} \to \prod_{\lambda|\ell} \operatorname{GL}_2 \left(\mathcal{O}_{\lambda}\right)
\]
whose components are the $\rho_{\lambda^\infty}$. It is also natural to consider $\lambda$-adic residual representations:
\begin{definition}\label{surf_def_GLambda}
If $\lambda$ is a place of $E$ above a good prime $\ell$ that is furthermore unramified in $E$, we write $G_\lambda$ for the image of the residual representation modulo $\lambda$, namely the image of the map $\rho_\lambda$ given by the composition
\[
\abGal{K} \xrightarrow{\rho_{\ell^\infty}} \prod_{\lambda|\ell} \operatorname{GL}_2 \left(\mathcal{O}_{\lambda}\right) \to \operatorname{GL}_2 \left(\mathcal{O}_{\lambda}\right) \to \operatorname{GL}_2 \left(\mathcal{O}_\lambda/\lambda\right).
\]
\end{definition}

The determinant of $\rho_{\lambda^\infty}$ is easy to describe:

\begin{lemma}{(\cite[Lemma 4.5.1]{MR0457455})}\label{surf_lemma_DescriptionDeterminant} For every place $\lambda$ of $E$ dividing a good prime $\ell$, the function
\[
\det \rho_{\lambda^\infty}: \abGal{K} \to \mathcal{O}_\lambda^\times
\]
coincides with the $\ell$-adic cyclotomic character $\chi_\ell:\abGal{K} \to \mathbb{Z}_\ell^\times \hookrightarrow \mathcal{O}_\lambda^\times$.
\end{lemma}

Observe that for a good prime $\ell$ the $\ell$-adic representation lands in $\operatorname{Aut}_{\mathcal{O}_\ell} \left( T_\ell (A) \right)$. If we regard $\mathbb{Z}_\ell$ as being embedded in $\mathcal{O}_\ell$ (the latter is naturally a $\mathbb{Z}_\ell$-algebra), then by the previous lemma we have $\det_{\mathcal{O}_\ell} \rho_{\ell^\infty}(g)=\chi_\ell(g) \in \mathbb{Z}_\ell^\times \subset \mathcal{O}_\ell^\times$, and applying lemma \ref{surf_lemma_Conclusion} we find
\begin{lemma}\label{surf_lemma_detIsSurjective}
If $\ell$ is good and unramified in $K$ then $\det_{\mathcal{O}_\ell} :G_{\ell^\infty} \to \mathbb{Z}_\ell^\times$ is surjective.
\end{lemma}

\subsection{The intersection $G_{\ell^\infty} \cap \operatorname{SL}_2(\mathcal{O}_\ell)$}
The key step in proving the surjectivity of the Galois representation for $\ell$ large enough lies in understanding the intersection $G_{\ell^\infty} \cap \operatorname{SL}_2(\mathcal{O}_\ell)$. A remarkable simplification of the problem comes from the fact that we only need to prove surjectivity for the residual mod-$\ell$ representation instead of the full $\ell$-adic system: this is made possible by the following `lifting' result, analogous to lemma \ref{surf_lemma_LiftingEndZ}.

\begin{proposition}{(\cite[Proposition 4.2]{MR1610883})}\label{surf_prop_LiftingLemma} Let $\mathcal{O}$ be the ring of integers of a number field $E$, $\lambda_1,\lambda_2, \ldots, \lambda_r$ distinct primes of $\mathcal{O}$ above $\ell$ and $H$ a closed subgroup of $\operatorname{SL}_2(\mathcal{O}_{\lambda_1}) \times \cdots \times \operatorname{SL}_2(\mathcal{O}_{\lambda_r})$ whose projection to $\operatorname{SL}_2(\mathbb{F}_{\lambda_1}) \times \cdots \times \operatorname{SL}_2(\mathbb{F}_{\lambda_r})$ is surjective. If $\ell$ is unramified in $E$ and $\ell \geq 5$, then $H$ is all of $\operatorname{SL}_2(\mathcal{O}_{\lambda_1}) \times \cdots \times \operatorname{SL}_2(\mathcal{O}_{\lambda_r})$. Under the same assumptions on $\ell$, if $G$ is a closed subgroup of $\operatorname{GL}_2(\mathcal{O}_{\lambda_1}) \times \cdots \times \operatorname{GL}_2(\mathcal{O}_{\lambda_r})$ whose projection to $\operatorname{GL}_2(\mathbb{F}_{\lambda_1}) \times \cdots \times \operatorname{GL}_2(\mathbb{F}_{\lambda_r})$ contains $\operatorname{SL}_2(\mathbb{F}_{\lambda_1}) \times \cdots \times \operatorname{SL}_2(\mathbb{F}_{\lambda_r})$, then $G'=\operatorname{SL}_2(\mathcal{O}_{\lambda_1}) \times \cdots \times \operatorname{SL}_2(\mathcal{O}_{\lambda_r})$.
\end{proposition}

\subsubsection{A little group theory}
We briefly review the group-theoretic results we are going to use, starting with the following sufficient criterion for a group to be a direct product:

\begin{lemma}{(\cite[Lemma 5.2.2]{MR0457455})}\label{surf_Ribet_ProductsOfTwo}
Let $S_1, \ldots, S_k$ ($k>1$) be finite groups with no nontrivial abelian quotients. Let $G$ be a subgroup of $S_1 \times \cdots \times S_k$ such that each projection $G \to S_i \times S_j$ ($1 \leq i < j \leq k$) is surjective. Then $G=S_1 \times \cdots \times S_k$.
\end{lemma}

We will also need the following version of \cite[Lemma 5.1]{MR1209248}; note that our formulation is slightly different than that of \cite{MR1209248}, but the same proof applies equally well.

\begin{lemma}{(\cite[Lemma 5.1]{MR1209248})}\label{surf_lemma_51}
Let $\ell \geq 5$ be a prime, $\mathbb{F}$ a finite field of characteristic $\ell$, and
\[
D=\left\{ (b,b') \in \operatorname{GL}_2(\mathbb{F}) \times \operatorname{GL}_2(\mathbb{F}) \bigm\vert \det(b)=\det(b') \in \mathbb{F}_\ell^\times \right\}.
\]

Let $H$ be a subgroup of $D$ whose projections on the two factors $\operatorname{GL}_2(\mathbb{F})$ contain $\operatorname{SL}_2(\mathbb{F})$. Then either $H$ contains $\operatorname{SL}_2(\mathbb{F}) \times \operatorname{SL}_2(\mathbb{F})$, or else there exist an isomorphism $f:\mathbb{F}^2 \to \mathbb{F}^2$ (of $\mathbb{F}$-vector spaces), a character $\chi:H \to \left\{ \pm 1 \right\}$, and a field automorphism $\sigma$ of $\mathbb{F}$ such that
\[
H \subseteq \left\{ (b,b') \in \operatorname{GL}_2(\mathbb{F}) \times \operatorname{GL}_2(\mathbb{F}) \bigm\vert b'=\chi((b,b')) \; \sigma \left(f b f^{-1}\right) \right\}.
\]
\end{lemma}

Finally, we will need a description of the subgroups of $\operatorname{GL}_2(\mathbb{F}_{\ell^\beta})$ for $\beta \geq 1$:

\begin{theorem}{(Dickson, \cite[Theorem 3.4]{MR0214671})}\label{surf_thm_Dickson}
Let $p$ be a prime number, $\beta$ be a positive integer, $q=p^\beta$ be a power of $p$, and $G$ be a subgroup of $\operatorname{GL}_2(\mathbb{F}_{q})$. Then, up to conjugacy in $\operatorname{GL}_2(\mathbb{F}_q)$, one of the following occurs:
\begin{enumerate}
\item $G$ is cyclic;
\item $G$ is a subgroup of the Borel group $\left\{\left(\begin{matrix} x & y \\ 0 & z  \end{matrix}  \right) \bigm\vert x,z \in \mathbb{F}_q^\times, y \in \mathbb{F}_q \right\}$;
\item $G$ contains (as a subgroup of index 2) a cyclic subgroup of order $u$, where $u$ divides $q^2-1$;
\item $G$ contains (as a subgroup of index 2) a subgroup consisting entirely of diagonal matrices;
\item[5a.] there is an $\alpha \in \mathbb{N}_{>0}$ dividing $\beta$ such that $G$ is generated by $\operatorname{SL}_2(\mathbb{F}_{p^\alpha})$ and by a scalar matrix $V$ (in this case $p^\alpha > 3$);
\item[5b.] there exist an $\alpha$ dividing $\beta$, a generator $\varepsilon$ of $\mathbb{F}_{p^\alpha}^\times$ (as a multiplicative group), and an element $b \in \mathbb{F}_{p^\beta}^\times$, such that $G$ is generated by $\operatorname{SL}_2(\mathbb{F}_{p^\alpha})$, a scalar matrix $V$, and the diagonal matrix $\operatorname{diag} \left(b, b \varepsilon \right)$; the subgroup generated by $\operatorname{SL}_2(\mathbb{F}_{p^\alpha})$ and $V$ is of type 5a, and has index 2 in $G$ (and again the inequality $p^\alpha>3$ holds);
\item[6.] $G/\left\{ \pm \operatorname{Id} \right\}$ is isomorphic to $S_4 \times \frac{\mathbb{Z}}{u \mathbb{Z}}$, $A_4 \times \frac{\mathbb{Z}}{u \mathbb{Z}}$ or $A_5 \times \frac{\mathbb{Z}}{u \mathbb{Z}}$, where $\frac{\mathbb{Z}}{u \mathbb{Z}}$ is identified with the subgroup generated by a scalar matrix in $\operatorname{GL}_2(\mathbb{F}_q)/\left\{ \pm \operatorname{Id} \right\}$.
\item[7.] $G$ is not of type (6), but $G/\left\{ \pm \operatorname{Id} \right\}$ contains $A_4 \times \frac{\mathbb{Z}}{u \mathbb{Z}}$ as a subgroup of index 2, and $A_4$ as a subgroup with cyclic quotient group; $\frac{\mathbb{Z}}{u \mathbb{Z}}$ is as in type (6) with $u$ even.
\end{enumerate}
\end{theorem}

\begin{definition}
In cases (5a) or (5b) the number $\alpha$ will be called the \textbf{level} of the group $G$.
\end{definition}

\subsubsection{Isogeny estimates}
Let $\lambda$ be a place of $E$. We write $\mathbb{F}_\lambda$ for the residue field at $\lambda$ and $f=[\mathbb{F}_\lambda:\mathbb{F}_\ell]$ for its inertia degree; 
recall that we have introduced the residual representation $G_\lambda$ in definition \ref{surf_def_GLambda}. The following two lemmas are simple variants of lemmas \ref{surf_lemma_Irreducibility} and \ref{surf_lemma_AbsoluteIrreducibility} respectively; we give a complete argument only for the former:

\begin{lemma}\label{surf_lemma_DimensionOne}
Let $\ell$ be a good prime for $A$ unramified in $E$, and let $\lambda$ be a place of $E$ above $\ell$. Suppose that $G_\lambda$ fixes a subspace $\Gamma$ of dimension 1 of $T_\lambda(A)/\lambda T_\lambda(A) \cong \mathbb{F}_\lambda^2$: then $\ell \leq b_0(A/K)$.
\end{lemma}

\begin{proof}
$\Gamma$ is fixed by $G_\lambda$ and therefore defined over $K$. Consider the $K$-variety $A^*=A/\Gamma$, which comes equipped with a natural isogeny $\pi: A \twoheadrightarrow A^*$ of degree $|\Gamma|=|\mathbb{F}_\lambda|=\ell^f$. Choose a $K$-isogeny $\psi:A^* \to A$ of degree $b \leq b_0(A/K)$. The composition $\psi \circ \pi$ is an endomorphism of $A$, so by hypothesis it is given by a certain $e \in \operatorname{End}(A) \subseteq \mathcal{O}_E$. Notice now that $e$ kills $\Gamma$, and on the other hand the action of $e$ on $T_\lambda(A)/\lambda T_\lambda(A)$ is through multiplication by the class of $e$ in $\mathcal{O}_E/\lambda \cong \mathbb{F}_\lambda$. It follows that $e$ reduces to 0 in $\mathbb{F}_\lambda$, that is, $e$ belongs to the ideal $\lambda$; hence $\deg(e)=N_{E/\mathbb{Q}}(e)^{2}$ (for this equality cf.~\cite{BirkenhakeLange}, Chapter 5, Corollary 1.3) is divisible by $N_{E/\mathbb{Q}}(\lambda)^{2}$, which is just $|\mathbb{F}_\lambda|^{2}=\ell^{2f}$. Comparing degrees we have $\ell^{2f} \bigm\vert \deg(e)=b\ell^f$, so $\ell^{f}$ divides $b$ which, in turn, is at most $b_0(A/K)$.
\end{proof}



\begin{lemma}\label{surf_lemma_DimensionTwo}
Let $\ell$ and $\lambda$ be as in the previous lemma. If $G_\lambda$ is abelian, then $\ell^2 \leq b_0\left(A^2/K\right)$.
\end{lemma}

The group $H_\ell$ of the following definition is the natural candidate for the image of $\rho_{\ell}$, for $\ell \gg 0$: it is the largest (connected) group whose elements are simultaneously symplectic similitudes for the Weil pairing and contained in the centralizer of the action of $E$.

\begin{definition}\label{surf_def_Hl}
Let $\ell$ be a good prime that does not ramify in $E$. We set
\[
H_\ell =\left\{ (h_\lambda)_{\lambda | \ell} \in \prod_{\lambda|\ell} \operatorname{GL}_2(\mathbb{F}_\lambda) \bigm\vert \det(h_{\lambda_1})=\det(h_{\lambda_2}) \in \mathbb{F}_\ell^\times \quad \forall \lambda_1, \lambda_2 | \ell \right\},
\]
where the product is over the places of $E$ that divide $\ell$.
\end{definition}
\begin{lemma}\label{surf_lemma_GlHl}
If $\ell$ is a good prime unramified in $E$, the group $G_\ell$ is contained in $H_\ell$.
\end{lemma}
\begin{proof}
The determinant of every $\rho_\lambda$ agrees with the cyclotomic character (lemma \ref{surf_lemma_DescriptionDeterminant}), so any two $h_\lambda$ have the same determinant.
\end{proof}

\medskip

To ease the notation we introduce the following definition:
\begin{definition}
Let $A/K$ be an abelian variety. We set
\[
M(A/K):=b(2[K:\mathbb{Q}],2\dim(A),2h(A))^{1/2}.
\]
\end{definition}

\begin{lemma}\label{surf_lemma_SL2Fl}
Let $\ell$ be a good prime, unramified in $E$ and strictly larger than $M(A/K)$, and let $\lambda$ be a place of $E$ above $\ell$. Up to conjugacy in $\operatorname{GL}_2(\mathbb{F}_\lambda)$, the group $G_\lambda$ contains $\operatorname{SL}_2(\mathbb{F}_\ell)$.
\end{lemma}

\begin{proof}
Suppose by contradiction that no $\operatorname{GL}_2(\mathbb{F}_\lambda)$-conjugate of $G_\lambda$ contains $\operatorname{SL}_2(\mathbb{F}_\ell)$: we shall prove that $\ell$ does not exceed $M(A/K)$.
By the Dickson classification (theorem \ref{surf_thm_Dickson}; cf.~also \cite[§2]{MR0387283}) we know that the following are the only possibilities (again, up to conjugacy in $\operatorname{GL}_2(\mathbb{F}_\lambda)$):

\begin{enumerate}[(I)]
\item $G_\lambda$ is contained in a Borel subgroup of $\operatorname{GL}_2(\mathbb{F}_\lambda)$: by definition, such a subgroup fixes a line, therefore $\ell \leq b(A/K)$ by lemma \ref{surf_lemma_DimensionOne}.

\item $G_\lambda$ is contained in the normalizer of a Cartan subgroup of $\operatorname{GL}_2(\mathbb{F}_\lambda)$: let $C$ be this Cartan subgroup and $N$ its normalizer. By the Dickson classification, the index $[N:C]$ is $2$, so the morphism
\[
\displaystyle \abGal{K} \to G_\lambda \to \frac{G_\lambda}{G_\lambda \cap C} \hookrightarrow \frac{N}{C}
\]
induces a quadratic character of $\abGal{K}$. Let $K'$ be the fixed field of the kernel of this character: we have $[K':K] \leq |N/C|=2$, and by construction the image of $\abGal{K'}$ in $\operatorname{Aut}\left(A[\lambda]\right)$ is contained in $C$. Applying lemma \ref{surf_lemma_DimensionTwo} to $A_{K'}$ we see that $\ell$ is at most $b\left(A^2/K'\right)^{1/2}$.

\item The projective image $\mathbb{P}G_\lambda$ of $G_\lambda$ is a finite group of order at most 60: by lemma \ref{surf_lemma_DimensionOne} we have $\ell \leq b(A/K'')$, where $K''$ is the fixed field of the kernel of $\abGal{K} \to G_\lambda \to \mathbb{P}G_\lambda$.
\end{enumerate}

It is clear that $G_\lambda$ does not fall in any of the previous cases -- and therefore contains $\operatorname{SL}_2(\mathbb{F}_\ell)$ up to conjugacy -- as long as $\ell$ is larger than $\max\left\{ b(A/K), b\left(A^2/K \right)^{1/2}, b(A^2/K')^{1/2}, b(A/K'') \right\}$. It is immediate to check that this maximum does not exceed $M(A/K)$.
\end{proof}

\medskip

Let $\ell$ be a good prime unramified in $E$, and write $\prod_{i=1}^n \lambda_i$ for its factorization in $\mathcal{O}_E$. We now aim to show that, for every $\lambda_i$ lying above $\ell$, the group $G_{\lambda_i}$ contains $\operatorname{SL}_2(\mathbb{F}_{\lambda_i})$.

Assume $\ell > M(A/K)$, so that by lemma \ref{surf_lemma_SL2Fl} we know that (up to conjugacy) every $G_{\lambda_i}$ contains $\operatorname{SL}_2(\mathbb{F}_\ell)$. We  set $\beta_i=[\mathbb{F}_{\lambda_i} : \mathbb{F}_\ell]$, so that $\ell^{\beta_i}$ is the order of the residue field at $\lambda_i$.
The assumption that $G_{\lambda_i}$ contains $\operatorname{SL}_2(\mathbb{F}_\ell)$ immediately implies that $G_{\lambda_i}$ must be of type (5a) or (5b) in the notation of theorem \ref{surf_thm_Dickson}. Suppose first $G_{\lambda_i}$ is of type (5a), generated (up to conjugacy) by $\operatorname{SL}_2\left( \mathbb{F}_{\ell^{\alpha_i}} \right)$ and by a scalar matrix $V=\mu \cdot \operatorname{Id}$. Since the determinant of any element in $G_{\lambda_i}$ lies in $\mathbb{F}_\ell^\times$ we know that $\det V=\mu^2$ is an element of $\mathbb{F}_{\ell}$, hence $V^2 \in \operatorname{GL}_2(\mathbb{F}_\ell)$. In particular, $G_{\lambda_i}$ contains as a subgroup of index 2 the group generated by $\operatorname{SL}_2\left( \mathbb{F}_{\ell^{\alpha_i}} \right)$ and $V^2$, which is a subgroup of $\operatorname{GL}_2\left(\mathbb{F}_{\ell^{\alpha_i}} \right)$. Furthermore, if $G_{\lambda_i}$ is of type (5b), then it contains a group of type (5a) as a subgroup of index 2. We thus deduce:
\begin{lemma}\label{surf_lemma_level}
Let $\ell$ be a good prime unramified in $E$, and suppose $\ell > M(A/K)$, so that $G_{\lambda_i}$ is of type (5a) or (5b). Let $\alpha_i$ be the level of $G_{\lambda_i}$. There exists an extension $K'$ of $K$, of degree at most 4, such that -- up to conjugacy -- the image of $
\rho_{\lambda_i} : \abGal{K'} \to \operatorname{GL}_2\left(\mathbb{F}_{\lambda_i} \right)$ is contained in $\operatorname{GL}_2\left(\mathbb{F}_{\ell^{\alpha_i}} \right)$.
\end{lemma}

Next we show that -- at least for $\ell$ large enough -- the level $\alpha_i$ must necessarily equal $\beta_i$, the degree $\left[\mathbb{F}_{\lambda_i}:\mathbb{F}_\ell\right]$:
\begin{lemma}\label{surf_lemma_SurjOnSL2Lambda}
Let $\ell$ be a good prime unramified in $E$. 
Suppose that, up to conjugacy in $\operatorname{GL}_2(\mathbb{F}_{\lambda_i})$, the group $G_{\lambda_i}$ is contained in $\operatorname{GL}_2\left(\mathbb{F}_{\ell^{\alpha_i}} \right)$ for some $\alpha_i < \beta_i$. Then $\ell \leq b_0(A/K)^{1/2}$.
\end{lemma}
\begin{proof}
For every place $\lambda$ of $E$ above $\ell$ we can identify $A[\lambda]$ with $\mathbb{F}_\lambda^2$, and for the factor $A[\lambda_i]$ the hypothesis allows us to choose coordinates in such a way that the image of $\rho_{\lambda_i}$ is contained $\operatorname{GL}_2\left(\mathbb{F}_{\ell^{\alpha_i}} \right)$. Consider now the subspace of $A[\ell]$ given by
\[
\Gamma=\left\{ (x_\lambda) \in \prod_{\lambda | \ell}A[\lambda] \cong \prod_{\lambda | \ell} \mathbb{F}_{\lambda}^2 \bigm\vert x_{\lambda_i} \in \left(\mathbb{F}_{\ell^{\alpha_i}}\right)^{2}, x_\lambda=0 \text{ for } \lambda \neq \lambda_i \right\}.
\]

By construction, $\Gamma$ is Galois-stable: indeed for any $g \in \abGal{K}$ and every $(x_\lambda) \in \Gamma$ we have
\[
\left(\rho_\lambda(g) \cdot x_{\lambda}\right)_{\lambda_i} = \rho_{\lambda_i}(g) \cdot x_{\lambda_i} \in \left(\mathbb{F}_{\ell^{\alpha_i}}\right)^2,
\]
since both the coefficients of the vector $x_{\lambda_i}$ and those of the matrix $\rho_{\lambda_i}(g)$ lie in $\mathbb{F}_{\ell^{\alpha_i}}$.
It follows that the abelian variety $A'=A/\Gamma$ is defined over $K$, and there are isogenies $\pi:A \to A'$ (the canonical projection, of degree $\ell^{2\alpha_i}$) and $\psi:A' \to A$ (which can be chosen to be of degree at most $b_0(A/K)$). Notice now that the endomorphism $e:=\psi \circ \pi \in \mathcal{O}_E$ of $A$ kills $\Gamma$; since the action of $e$ on $A[\lambda_i]$ is given by the class $[e]$ of $e$ in $\mathbb{F}_{\lambda_i}$, it follows that $[e]=0$, that is, $e$ belongs to the ideal $\lambda_i$. Thus we see that the degree of $\psi \circ \pi$ satisfies
\[
\ell^{2\beta_i} = N_{E/\mathbb{Q}}\left( \lambda_i \right)^2 \bigm\vert N_{E/\mathbb{Q}}\left( e \right)^2 \bigm\vert \operatorname{deg}\left(\psi \circ \pi\right) = \ell^{2\alpha_i} \operatorname{deg} \psi \leq \ell^{2\alpha_i} b_0(A/K),
\]
and therefore we have $\ell^2 \leq \ell^{2(\beta_i-\alpha_i)} \leq b_0(A/K)$. 
\end{proof}

\medskip

Combining the previous two lemmas we find
\begin{corollary}\label{surf_cor_TrueSurjOnOneFactor}
Let $\ell > M(A/K)$ be a good prime unramified in $E$, and let $\lambda$ be a place of $E$ above $\ell$. The image of the representation $\rho_{\lambda} : \abGal{K} \to \operatorname{GL}_2(\mathbb{F}_\lambda)$ contains $\operatorname{SL}_2(\mathbb{F}_\lambda)$.
\end{corollary}
\begin{proof}
By lemma \ref{surf_lemma_SL2Fl} we know that $G_\lambda$ is of type (5a) or (5b) in the sense of theorem \ref{surf_thm_Dickson}. Let $\alpha$ be the level of $G_\lambda$ and $\beta=[\mathbb{F}_\lambda:\mathbb{F}_\ell]$; it is clear that it is enough to show $\alpha=\beta$. By lemma \ref{surf_lemma_level}, passing to an extension $K'$ of $K$ of degree at most 4 we can assume that (up to conjugacy) the image of $\rho_{\lambda} : \abGal{K} \to \operatorname{Aut} A[\lambda]$ is contained in $\operatorname{GL}_2\left(\mathbb{F}_{\ell^{\alpha}} \right)$. The corollary then follows from lemma \ref{surf_lemma_SurjOnSL2Lambda} (applied to $A/K'$) and the obvious inequality $M(A/K) > b(A/K')^{1/2} \geq b_0(A/K')^{1/2}$.
\end{proof}

\smallskip

\begin{lemma}\label{surf_lemma_SurjOnTwo}
Let $\lambda_1,\lambda_2$ be two places of $\mathcal{O}_E$ above the good prime $\ell$. Suppose $\ell$ is larger than $M(A/K)$ and unramified in $E$: then the image of $\abGal{K} \xrightarrow{\; \rho_{\lambda_1} \times \rho_{\lambda_2} \;} G_{\lambda_1} \times G_{\lambda_2}$ contains $\operatorname{SL}_2(\mathbb{F}_{\lambda_1}) \times \operatorname{SL}_2(\mathbb{F}_{\lambda_2})$.
\end{lemma}

\begin{proof}
Let $S$ be the image of $G_\ell$ in $G_{\lambda_1} \times G_{\lambda_2}$. For the sake of simplicity write $S_i=\operatorname{SL}_2(\mathbb{F}_{\lambda_i})$ for $i=1,2$ and set $S^{1}=S \cap \left( S_1 \times S_2 \right)$; we claim that $S^1$ projects surjectively onto $S_1$ and $S_2$. To see this, notice first that we know from corollary \ref{surf_cor_TrueSurjOnOneFactor} that for every $g_1 \in \operatorname{SL}_2(\mathbb{F}_{\lambda_1})$ there exists in $S$ an element of the form $(g_1,g_2)$, for some $g_2 \in G_{\lambda_2}$. On the other hand, we know from lemma \ref{surf_lemma_DescriptionDeterminant} that $\det(g_2)=\det(g_1)=1$, so $(g_1,g_2)$ actually belongs to $S^1$, which therefore projects surjectively onto $\operatorname{SL}_2(\mathbb{F}_{\lambda_1})$. By the same argument,  $S^1 \to \operatorname{SL}_2(\mathbb{F}_{\lambda_2})$ is onto as well.

The statement of the lemma is equivalent to $S^1 =S_1 \times S_2$. Suppose by contradiction that this is not the case: then by Goursat's lemma there exist normal subgroups $N_1, N_2$ (of $S_1, S_2$ respectively) and an isomorphism $\varphi: S_1/N_1 \to S_2/N_2$ such that $S^1$ projects to the graph of $\varphi$ in $S_1/N_1 \times S_2/N_2$. Comparing the orders of $S_1/N_1$ and $S_2/N_2$ (or, more precisely, their valuations at $\ell$) easily gives $\mathbb{F}_{\lambda_1}\cong \mathbb{F}_{\lambda_2}$. We now go back to working directly with the group $S$. By lemma \ref{surf_lemma_51}, which is applicable since the characteristic of $\mathbb{F}_{\lambda_1} \cong \mathbb{F}_{\lambda_2}$ is $\ell \geq 5$, there exist an isomorphism $f:\mathbb{F}_{\lambda_1}^2 \to \mathbb{F}_{\lambda_2}^2$, a character $\chi:S \to \left\{\pm 1\right\}$, and an automorphism $\sigma$ of $\mathbb{F}_{\lambda_1}=\mathbb{F}_{\lambda_2}$ such that $g_2=\chi((g_1,g_2)) \sigma \left(f g_1 f^{-1}\right)$ for all $(g_1,g_2)$ in $S$. Assume first that $\chi$ is the trivial character: then the subspace
\[
\Gamma:=\left\{ (x_\lambda) \in \prod_{\lambda | \ell} \mathbb{F}_\lambda^2  \cong \prod_{\lambda | \ell} A[\lambda] \bigm\vert x_{\lambda_2}=\sigma(f x_{\lambda_1}), \; x_\lambda=0 \text{ for } \lambda \neq \lambda_1,\lambda_2 \right\}
\]
is Galois invariant, so the abelian variety $A^*:=A/\Gamma$ is defined over $K$. Let $\pi: A \to A^*$ be the canonical projection and $\psi:A^* \to A$ be an isogeny of degree at most $b_0(A/K)$. As in the previous lemma, we see that since $e:=\psi \circ \pi \in \mathcal{O}_E$ kills $\Gamma$ we must have $e \in \lambda_1$ and $e \in \lambda_2$. Thus if $\beta$ denotes the common degree $[\mathbb{F}_{\lambda_1}:\mathbb{F}_\ell]=[\mathbb{F}_{\lambda_2}:\mathbb{F}_\ell]$ we have
\[
\ell^{4\beta}=N_{E/\mathbb{Q}}\left( \lambda_1 \lambda_2 \right)^2 \leq \deg e = \deg \left( \psi \circ \pi \right) = \ell^{2\beta} \deg \psi \leq \ell^{2\beta} b_0(A/K),
\]
whence $\ell \leq b_0(A/K)^{1/2}< M(A/K)$, contradiction. If, on the other hand, $\chi$ is not the trivial character, then the kernel of $\abGal{K} \to G_\ell \to S \xrightarrow{\chi} \left\{\pm 1 \right\}$ defines a quadratic extension $K'$ of $K$, and repeating the same argument over $K'$ we find $\ell \leq b_0(A/K;2)^{1/2}<M(A/K)$, which is again a contradiction.
\end{proof}

\smallskip

We are now ready to prove theorem \ref{surf_thm_GL2}, whose statement we reproduce here for the reader's convenience:

\begin{theorem}{(Theorem \ref{surf_thm_GL2})}
Let $A/K$ be an abelian variety of dimension $g$. Suppose that $\operatorname{End}_{\overline{K}}(A)$ is an order in a totally real field $E$ of degree $g$ over $\mathbb{Q}$ (that is, $A$ is of $\operatorname{GL}_2$-type), and that all the endomorphisms of $A$ are defined over $K$. Let $\ell$ be a prime unramified both in $K$ and in $E$ and strictly larger than 
$
\max\left\{b(A/K)^{g}, \; b(2[K:\mathbb{Q}],2\dim(A),2h(A))^{1/2} \right\}
$: then we have
\[
G_{\ell^\infty}=\left\{x \in \operatorname{GL}_2\left(\mathcal{O}_E \otimes \mathbb{Z}_\ell \right) \bigm\vert \operatorname{det}_{\mathcal{O}_E} x \in \mathbb{Z}_\ell^\times \right\}.
\]
\end{theorem}

\begin{proof}
Notice first that $\ell$ is a good prime by proposition \ref{surf_prop_EveryPrimeIsGood}. By lemma \ref{surf_lemma_SurjOnTwo}, the inequality $\ell>M(A/K)$ guarantees that for every pair of places $\lambda_1,\lambda_2$ of $E$ lying above $\ell$ the image of
\[
\abGal{K} \xrightarrow{\rho_{\lambda_1} \times \rho_{\lambda_2}} \operatorname{GL}_2(\mathbb{F}_{\lambda_1}) \times \operatorname{GL}_2(\mathbb{F}_{\lambda_2})
\]
contains $\operatorname{SL}_2(\mathbb{F}_{\lambda_1}) \times \operatorname{SL}_2(\mathbb{F}_{\lambda_2})$. Let $S:=G_\ell \cap \prod_{\lambda | \ell} \operatorname{SL}_2(\mathbb{F}_\lambda)$. By lemma \ref{surf_lemma_GlHl} and what we just remarked we see that for every pair $\lambda_1, \lambda_2$ of places of $E$ dividing $\ell$ the group $S$ projects surjectively onto $\operatorname{SL}_2(\mathbb{F}_{\lambda_1}) \times \operatorname{SL}_2(\mathbb{F}_{\lambda_2})$.
Clearly we have $\ell \geq 5$, so a group of the form $\operatorname{SL}_2(\mathbb{F}_{\lambda})$ has no nontrivial abelian quotients: lemma \ref{surf_Ribet_ProductsOfTwo} then implies $S=\prod_{\lambda |\ell} \operatorname{SL}_2(\mathbb{F}_\lambda)$, and by proposition \ref{surf_prop_LiftingLemma} the group $G_{\ell^\infty}$ contains $\operatorname{SL}_2(\mathcal{O}_E \otimes \mathbb{Z}_\ell)$. Since furthermore the map $\det_{\mathcal{O}_E \otimes \mathbb{Z}_\ell} : G_{\ell^\infty} \to \mathbb{Z}_\ell^\times$ is surjective by lemma \ref{surf_lemma_detIsSurjective} we conclude that $G_{\ell^\infty}$ contains $\left\{x \in \operatorname{GL}_2\left(\mathcal{O}_E \otimes \mathbb{Z}_\ell \right) \bigm\vert \operatorname{det}_{\mathcal{O}_E} x \in \mathbb{Z}_\ell^\times \right\}$. The opposite inclusion is proved as in lemma \ref{surf_lemma_GlHl}, so these two groups are equal as claimed.
\end{proof}

\begin{remark}
It is not hard to show that when $g$ is large enough we have $M(A/K) < b(A/K)^{g}$; in fact, $g \geq 33$ suffices.
\end{remark}

The case of abelian surfaces follows at once:

\begin{corollary}{(Corollary \ref{surf_cor_MainRM})}
Suppose that $R=\operatorname{End}_{\overline{K}}(A)$ is an order in a real quadratic field $E$ and that all the endomorphisms of $A$ are defined over $K$. Let $\ell$ be a prime unramified both in $K$ and in $E$ and strictly larger than $b(2[K:\mathbb{Q}],4,2h(A))^{1/2}$: then we have
\[
G_{\ell^\infty}=\left\{x \in \operatorname{GL}_2\left(\mathcal{O}_E \otimes \mathbb{Z}_\ell \right) \bigm\vert \operatorname{det}_{\mathcal{O}_E} x \in \mathbb{Z}_\ell^\times \right\}.
\]
\end{corollary}
\begin{proof}
Immediate from the previous theorem and the (easy) inequality $M(A/K) > b(A/K)^{2}$.
\end{proof}

\section{Type II -- Quaternionic multiplication}\label{surf_sect_QM}
In this section we establish the surjectivity result when $R=\operatorname{End}_{\overline{K}}(A)$ is an order in an indefinite (division) quaternion algebra $D$ over $\mathbb{Q}$, and the action of $R$ is defined over $K$. We let $\Delta$ be the discriminant of $R$, and we fix a maximal order $\mathcal{O}_D$ of $D$ which contains $R$. For the whole section we only consider primes $\ell$ strictly larger than $b(2[K:\mathbb{Q}],4,2h(A))^{1/2}$.

\begin{remark}\label{surf_rmk_Index}
We know from proposition \ref{surf_prop_EndomorphismRingIndex} that the index $[\mathcal{O}_D:R]$ does not exceed $b(A/K)^4$. It is easy to check that $b(A/K)^4<b(2[K:\mathbb{Q}],4,2h(A))^{1/2}$, and thus for all the primes $\ell$ we consider we have $\ell \nmid [\mathcal{O}_D:R]$.
\end{remark}

We start by recalling a result from \cite{MR2290584}, which we state only in the special case we need (abelian surfaces with quaternionic multiplication): 

\begin{theorem}{(\cite[Theorem 5.4]{MR2290584})}\label{surf_thm_ReducedWeilPairing}
Let $\mathcal{O}_D$ be a maximal order of $\operatorname{End}_{\overline{K}}(A)$ containing $R$. Let $\ell$ be a prime dividing neither $\Delta$ nor the index $[\mathcal{O}_D:R]$. Suppose that $\ell$ does not divide the degree of a fixed $K$-polarization of $A$. There exists a $\abGal{K}$-equivariant isomorphism
\[
T_\ell(A) \cong W_{\ell^\infty} \oplus W_{\ell^\infty},
\]
where $W_{\ell^\infty}$ is a simple $\abGal{K}$-module, free of rank 2 over $\mathbb{Z}_\ell$, equipped with a nondegenerate, $\abGal{K}$-equivariant bilinear form
\[\langle \cdot, \cdot \rangle_{QM} : W_{\ell^\infty} \times W_{\ell^\infty} \to \mathbb{Z}_\ell(1).\]
\end{theorem}

\noindent\textbf{Notation.} We write $W_\ell$ for $W_{\ell^\infty}/\ell W_{\ell^\infty}$. It is a $\abGal{K}$-module, free of rank 2 over $\mathbb{F}_\ell$.

\medskip

Theorem \ref{surf_thm_ReducedWeilPairing} says in particular that $T_\ell A$ decomposes as the direct sum of two isomorphic $2$-dimensional representations. Thus choosing bases for $W_{\ell^\infty}$ and $W_\ell$ we have:

\begin{lemma}
If $\ell$ does not divide $\Delta$ nor $[\mathcal{O}_D:R]$, then $G_\ell$ can be identified with a subgroup of $\operatorname{GL}_2(\mathbb{F}_\ell)$ acting on $A[\ell] \cong \mathbb{F}_\ell^4 \cong \operatorname{M}_2(\mathbb{F}_\ell)$ on the right. Similarly, $G_{\ell^\infty}$ can be identified with a subgroup of $\operatorname{GL}_2(\mathbb{Z}_\ell)$ acting on $T_\ell(A) \cong \mathbb{Z}_\ell^4 \cong \operatorname{M}_2(\mathbb{Z}_\ell)$ on the right.
\end{lemma}

From now on we use the description of $G_\ell$ given by the previous lemma, namely we consider it as a subgroup of $\operatorname{GL}_2(\mathbb{F}_\ell)$ acting on $\mathbb{F}_\ell^4$ as two copies of the standard representation. 

\begin{lemma}
Let $\ell$ be a prime which divides neither $\Delta$ nor $[\mathcal{O}_D:R]$. Under the above identification $T_\ell(A) \cong \operatorname{M}_2(\mathbb{Z}_\ell)$, the action of $R \otimes \mathbb{Z}_\ell \cong \operatorname{M}_2(\mathbb{Z}_\ell)$ on $T_\ell(A)$ is the natural multiplication of matrices (with $R \otimes \mathbb{Z}_\ell$ acting on the left).
\end{lemma}
\begin{proof}
Notice first that $R \otimes \mathbb{Z}_\ell \cong \operatorname{M}_2(\mathbb{Z}_\ell)$ since $\ell \nmid \Delta$. Next by \cite[Theorem 7.3]{MR1156568} we know that $G_{\ell^\infty}$ is open in $\operatorname{GL}_2(\mathbb{Z}_\ell)$; in particular, $G_{\ell^\infty}$ is Zariski-dense in $\operatorname{M}_2(\mathbb{Z}_\ell)$. Let $\varphi$ be any element of $R \otimes \mathbb{Z}_\ell$: then $G_{\ell^\infty}$ commutes with $\varphi$, and since the property of commuting with a fixed element is Zariski-closed we see that $\varphi$ commutes with \textit{any} element of $\operatorname{M}_2(\mathbb{Z}_\ell)$ (acting on $T_\ell(A) \cong \operatorname{M}_2(\mathbb{Z}_\ell)$ by right multiplication). Since this holds for every $\varphi \in R \otimes \mathbb{Z}_\ell$, it follows that the left action of $R \otimes \mathbb{Z}_\ell$ on $T_\ell(A)$ commutes with the right action of $\operatorname{M}_2(\mathbb{Z}_\ell)$. Since $T_\ell(A)$ is a free one-dimensional right module over $\operatorname{M}_2(\mathbb{Z}_\ell)$, this implies that $\varphi$ acts on $T_\ell(A)$ as \textit{left} multiplication by an element of $\operatorname{M}_2(\mathbb{Z}_\ell)$.
\end{proof}

\begin{lemma}\label{surf_lemma_QMSL2} Suppose $\ell$ divides neither $\Delta$ nor $[\mathcal{O}_D:R]$ and is larger than $b(2[K:\mathbb{Q}],4,2h(A))^{1/2}$. The group $G_\ell$ contains $\operatorname{SL}_2(\mathbb{F}_\ell)$ under the above identification.
\end{lemma}

\begin{proof}
This is a very minor variant of lemmas \ref{surf_lemma_DimensionOne} and \ref{surf_lemma_DimensionTwo}, so we only sketch the proof. If $G_\ell$ does not contain $\operatorname{SL}_2(\mathbb{F}_\ell)$, then Dickson's classification (theorem \ref{surf_thm_Dickson}) implies that one of the following holds:
\begin{itemize}
\item $G_\ell$ is contained in a Borel subgroup: we can find a line $\Gamma \subseteq W_\ell$ that is stable under the action of $G_\ell$. Using the previous lemma and applying an obvious variant of the argument of lemma \ref{surf_lemma_DimensionOne} to the isogeny $\displaystyle A \to \frac{A}{\Gamma \oplus \Gamma}$ we find $\ell^2 \leq b_0(A/K) \leq b(A/K)$.

\item The projective image of $G_\ell$ has cardinality at most 60: upon replacement of $K$ with an extension of degree at most 60 we are back to the previous case, hence we find that the inequality $\ell^2 \leq b(60[K:\mathbb{Q}],2,h(A))$ holds.

\item Up to replacing $K$ with an extension $K'$ of degree at most 2, $G_\ell$ is commutative, but does not entirely consist of scalars (this case being covered by the first one). We can choose an $\alpha \in G_\ell$ which is not a scalar, and apply a variant the argument of lemma \ref{surf_lemma_DimensionTwo} to the isogeny given by the natural projection from $A \times A$ to its quotient by the subgroup
\[
\left\{ (x_1,y_1,x_2,y_2) \in W_\ell \oplus W_\ell \oplus W_\ell \oplus W_\ell \cong A[\ell] \times A[\ell] \bigm\vert x_2=\alpha x_1, \; y_1=y_2=0 \right\}.
\]

The conclusion is now $\ell^2 \leq b_0(A^2/K') \leq b(2[K:\mathbb{Q}],4,2h(A))$.

\end{itemize}

Comparing the various bounds thus obtained we see that $b(2[K:\mathbb{Q}],4,2h(A))^{1/2}$ is much larger than any of the others, thus establishing the lemma.
\end{proof}

\begin{lemma}
Suppose $\ell$ is a prime that divides neither $\Delta$ nor $[\mathcal{O}_D:R]$, so that $R \otimes \mathbb{Z}_\ell \cong \operatorname{M}_2(\mathbb{Z}_\ell)$. Suppose furthermore that $\ell$ does not divide the degree of a minimal $K$-polarization of $A$. For every $g \in \abGal{K}$ the determinant of $\rho_\ell(g)$, thought of as an element of $\operatorname{GL}_2(\mathbb{Z}_\ell)$ (and not of $\operatorname{GSp}(T_\ell(A))$), is $\chi_\ell(g)$.
\end{lemma}

\begin{proof}
This is the same argument as for elliptic curves. If we fix a basis $e_1, e_2$ of $W_{\ell^\infty}$ and write $\left( \begin{matrix} a & b\\ c & d \end{matrix} \right)$ for the matrix representing the action of $\rho_{\ell^\infty}(g)$ in this basis, we obtain
\[
\begin{aligned}
\chi_\ell(g) \langle e_{1}, e_{2} \rangle_{QM} & = \langle \rho_{\ell^\infty}(g) e_{1}, \rho_{\ell^\infty}(g) e_{2}  \rangle_{QM}\\
                                               & =\langle a e_{1} + ce_{2}, be_{1}+de_{2} \rangle_{QM} \\
																							 & =ad \langle e_{1}, e_{2} \rangle+bc \langle e_{2}, e_{1} \rangle_{QM} \\
																							 & =(ad-bc)\langle e_{1}, e_{2} \rangle_{QM},
\end{aligned}
\]
and since $\langle e_{1}, e_{2} \rangle_{QM}$ does not vanish we obtain $\chi_\ell(g) = (ad-bc)=\det \rho_\ell(g)$ as claimed.
\end{proof}


\begin{theorem}{(Theorem \ref{surf_thm_MainQM})}
Let $A/K$ be an abelian surface such that $R=\operatorname{End}_{\overline{K}}(A)$ is an order in an indefinite quaternion division algebra over $\mathbb{Q}$, and let $\Delta$ be the discriminant of $R$. Suppose that all the endomorphisms of $A$ are defined over $K$. If $\ell$ is larger than $b(2[K:\mathbb{Q}],4,2h(A))^{1/2}$, does not divide $\Delta$, and is unramified in $K$, then the equality $G_{\ell^\infty}=\left(R \otimes \mathbb{Z}_\ell\right)^\times$ holds.
\end{theorem}

\begin{proof}
As $b(2[K:\mathbb{Q}],4,2h(A))^{1/2}>b(A/K)$, by \cite[Théorème 1]{PolarisationsEtIsogenies} we see that $\ell$ does not divide the degree of a minimal polarization of $A$. Combining this fact with remark \ref{surf_rmk_Index}, we see from theorem \ref{surf_thm_ReducedWeilPairing} that there exist well-defined modules $W_{\ell^\infty},W_\ell$ and a nondegenerate bilinear form $\langle \cdot, \cdot \rangle_{QM}$ on $W_{\ell^\infty}$.

By lemma \ref{surf_lemma_QMSL2} the inequality imposed on $\ell$ guarantees that $G_\ell$ contains $\operatorname{SL}_2(\mathbb{F}_\ell)$. It follows that $G_{\ell^\infty}$ is a closed subgroup of $\left(R \otimes_{\mathbb{Z}} \mathbb{Z}_\ell\right)^\times \cong \operatorname{GL}_2(\mathbb{Z}_\ell)$ whose projection modulo $\ell$ contains $\operatorname{SL}_2(\mathbb{F}_\ell)$. Since we certainly have $\ell \geq 5$, it follows from lemma \ref{surf_lemma_LiftingEndZ} that $G_{\ell^\infty}$ contains $\operatorname{SL}_2(\mathbb{Z}_\ell)$. On the other hand, the previous lemma and the condition that $\ell$ is unramified in $K$ ensure that $\det : G_{\ell^\infty} \to \mathbb{Z}_\ell^\times$ is onto, so $G_{\ell^\infty}=\operatorname{GL}_2(\mathbb{Z}_\ell)$ as claimed.
\end{proof}




\medskip

Let us make a few closing remarks on this case. It is a general philosophy that -- at the level of Galois representations -- an abelian variety of dimension $2g$ with quaternionic multiplication by an algebra $D$ (whose center is the number field $L$) should behave like two copies of an abelian variety of dimension $g$ and endomorphism algebra $L$. The proof we have just given shows that this philosophy is very much correct in the case of surfaces, and indeed from lemma \ref{surf_lemma_QMSL2} onward this is virtually the same proof as for elliptic curves (cf.~for example \cite{MR1209248}).
Even more precisely, write the bound we obtained for a QM surface in the form $b(2[K:\mathbb{Q}],2\dim(A),2h(A))^{1/2}$; for an elliptic curve $E/K$ without (potential) complex multiplication, the Galois representation is surjective onto $\operatorname{GL}_2(\mathbb{Z}_\ell)$ for every prime $\ell$ that does not ramify in $K$ and is larger than $b(2[K:\mathbb{Q}],2\dim E,2h(E))^{1/2}$ (cf.~\cite{MR1209248}), which is formally the same expression. On the other hand, the actual numerical dependence of the present result on the height of $A$ is much worse than the analogous one for elliptic curves, due to the strong dependence of the function $b([K:\mathbb{Q}],\dim A,h(A))$ on the parameter $\dim A$.

\begin{remark}
In the light of this discussion, it is reasonable to think that the methods of \cite{AdelicEC} might be generalized to give results on the index of the \textit{adelic} representation attached to $A$. We do not attempt this here, for doing so would entail giving a classification of the integral Lie subalgebras of \textit{any} $\mathbb{Z}_\ell$-form of $\mathfrak{sl}_2$: indeed, these algebras appear when we try to study the precise structure of $G_{\ell^\infty}$ for those $\ell$ that divide $\Delta$. Such a classification task seems rather daunting, given that the easier problem of studying the $\mathbb{Q}_\ell$-forms of $\mathfrak{sl}_2$ is already highly nontrivial.
\end{remark}

\appendix

\section{The index of the endomorphism ring}\label{surf_sect_Index}

Let $A/K$ be an absolutely simple abelian variety. Its endomorphism ring $R=\operatorname{End}_{\overline{K}}(A)$ is an order in a finite-dimensional division algebra $D$ over $\mathbb{Q}$, and we are interested in giving a bound on the index of $R$ in any maximal order $\mathcal{O}_D$ containing it. Note that when $D$ is a field there is a unique maximal order, namely its ring of integers, but when $D$ is not commutative the index $[\mathcal{O}_D:R]$ might a priori depend on the choice of $\mathcal{O}_D$. The following proposition shows that this is not the case:

\begin{proposition}
Let $L$ be a number field, $D$ a central simple algebra over $L$, and $R$ an order of $D$. Let $\mathcal{O}_D$ be a maximal order in $D$ containing $R$. The index $[\mathcal{O}_D:R]$ does not depend on the choice of $\mathcal{O}_D$.
\end{proposition}

\begin{proof}
Note first that any maximal order of $D$ is stable under multiplication by $\mathcal{O}_L$ (indeed if $S$ is a subring of $D$ then the $\mathcal{O}_L$-module generated by $S$ is again a subring of $D$), so the order $R'$ generated by $R$ and $\mathcal{O}_L$ is again contained in $\mathcal{O}_D$. We have $[\mathcal{O}_D:R]=[\mathcal{O}_D:R'][R':R]$, and since $[R':R]$ clearly does not depend on $\mathcal{O}_D$ we can assume that $R=R'$, i.e.~that $R$ is an $\mathcal{O}_L$-order. Under this additional assumption we have
\[
\mathcal{O}_D/R \cong  \bigoplus_{v \text{ finite place of } L}  \frac{\mathcal{O}_D \otimes_{\mathcal{O}_L} {\mathcal{O}_L}_v}{R \otimes_{\mathcal{O}_L} {\mathcal{O}_L}_v},
\]
so that  $[\mathcal{O}_D:R]=\prod_{v \text{ finite place of } L} [\mathcal{O}_D \otimes {\mathcal{O}_L}_v : R_v]$, where $R_v = R \otimes_{\mathcal{O}_L} {\mathcal{O}_L}_v$. The order $\mathcal{O}_D \otimes {\mathcal{O}_L}_v$ is maximal in $\mathcal{O}_D \otimes L_v$ (\cite[Corollary 11.2 and Theorem 11.5]{MR1972204}), so it is enough to prove that at every finite place the index $[\mathcal{O}_D \otimes {\mathcal{O}_L}_v : R_v]$ is independent of the choice of $\mathcal{O}_D$. We are thus reduced to the complete local case: by Theorem 17.3 of \cite{MR1972204} all maximal orders in $\mathcal{O}_D \otimes L_v$ are conjugated. We now write the index $[\mathcal{O}_D \otimes {\mathcal{O}_L}_v : R_v]$ as the ratio $\dfrac{\operatorname{covol}(R_v)}{\operatorname{covol}(\mathcal{O}_D \otimes {\mathcal{O}_L}_v)}$, where the covolume is taken with respect to any Haar measure (on $\mathcal{O}_D \otimes L_v$): as the Haar measure is invariant under conjugation, this quantity does not depend on $\mathcal{O}_D$.
\end{proof}

\medskip

In order to simplify matters it is convenient to assume that all the endomorphisms of $A$ are defined over $K$. This condition is completely harmless, since it can be achieved at the expenses of a controllable extension of $K$:

\begin{lemma}{(\cite[Theorem 4.1]{MR1154704})}\label{surf_lemma_KIsIrrelevant}
There exists a number field $K'$, with $[K':K]$ bounded only in terms of $g=\dim(A)$, such that all the endomorphisms of $A$ are defined over $K'$. We can take $[K':K] \leq 2 \cdot (9g)^{2g}$. 
\end{lemma}

From now on we will therefore assume that all the endomorphisms of $A$ are defined over $K$. In order to get estimates in the case of noncommutative endomorphism algebras we will need the following lemma, which is essentially \cite[Proposition 2.5.4]{CurvesOfGenus2}: even though the latter was stated only for commutative endormorphism rings, the proof works equally well in the general case. 

\begin{lemma}{(\cite[Proposition 2.5.4]{CurvesOfGenus2})}\label{surf_lemma_CanExtendEndRing}
Let $D$ be a division algebra, $R \subseteq S$ be orders in $D$ and $A/K$ be an abelian variety with $\operatorname{End}_K(A)=R$. There exists an abelian variety $B/K$, isogenous to $A$ over $K$, such that $\operatorname{End}_K(B) \supseteq S$.
\end{lemma}

\begin{corollary}
Let $A/K$ be an abelian variety with endomorphism ring $R$. Write $D=R \otimes \mathbb{Q}$, and let $\mathcal{O}_D$ be any maximal order of $D$ containing $R$. Suppose that all the endomorphisms of $A$ are defined over $K$. There exists an abelian variety $A'/K$ and two isogenies $\varepsilon_1:A \to A', \varepsilon_2:A' \to A$, defined over $K$, such that $\operatorname{End}_K(A')= \mathcal{O}_D$ and
$\max\left\{\operatorname{deg}(\varepsilon_1), \operatorname{deg}(\varepsilon_2)\right\} \leq b(A/K)$.
\end{corollary}

\begin{proof}
Lemma \ref{surf_lemma_CanExtendEndRing} shows the existence of a $K$-variety $A'$ having $\mathcal{O}_D$ as its endomorphism ring, so the claim follows from \cite[Theorem 1.4]{PolarisationsEtIsogenies} (which is a symmetric version of theorem \ref{surf_thm_Isogeny}, bounding the degrees of minimal isogenies both from $A$ to $A'$ and from $A'$ to $A$).
\end{proof}

\medskip

We can now deduce the desired bound on $[\mathcal{O}_D : R]$:
\begin{proposition}\label{surf_prop_EndomorphismRingIndex}
The inequality $[\mathcal{O}_D : R] \leq b(A/K)^{\dim_\mathbb{Q}(D)}$ holds.
\end{proposition}

\begin{proof}
Let $A', \varepsilon_1, \varepsilon_2$ be as in the above corollary. Consider the following linear map:
\[
\begin{matrix}
\varphi: & \operatorname{End}_K(A') & \to & \operatorname{End}_K(A) & \hookrightarrow & \operatorname{End}_K(A') \\
         &        e            & \mapsto & \varepsilon_2 \circ e \circ \varepsilon_1,
\end{matrix}
\]
where the second embedding is given by the fact that $R=\operatorname{End}_K(A)$ is an order in $D$ and $\operatorname{End}_K(A')=\mathcal{O}_D$ is a maximal order containing $R$. Note that $\operatorname{End}_K(A)$ is endowed with a positive-definite quadratic form given by the degree. We consider both $\operatorname{End}_K(A)$ and $\operatorname{End}_K(A')$ as lattices sitting inside $D_\mathbb{R}=\operatorname{End}_K(A') \otimes_{\mathbb{Z}} \mathbb{R}$, and observe that the degree map extends naturally to a positive-definite quadratic form on $D_\mathbb{R}$. This makes $D_\mathbb{R}$ into an Euclidean space, which in particular comes equipped with a natural (Lebesgue) measure. Denote by $r$ the $\mathbb{R}$-dimension of $D_\mathbb{R}$, which is also the dimension of $D$ as a $\mathbb{Q}$-vector space.
Since the equality $\deg(e_1 \circ e_2)=\deg(e_1) \cdot \deg(e_2)$ holds for any pair of isogenies between abelian varieties, we have
\[
\deg(\varphi(e))=\deg(\varepsilon_2 \circ e \circ \varepsilon_1)=\deg(\varepsilon_1)\deg(\varepsilon_2)\deg(e) \leq b(A/K)^2 \deg(e).
\]

Extend $\varphi$ by linearity to an endomorphism of $D_\mathbb{R}$ (which we still denote by $\varphi$) and fix a $\deg$-orthonormal basis $\gamma_1, \ldots, \gamma_r$ of $D_\mathbb{R}$. By construction the inclusion $\varphi(\mathcal{O}_D) \subseteq R$ holds, so we have the inequality
\[
[\mathcal{O}_D : R] = \frac{\operatorname{covol}(R)}{\operatorname{covol}(\mathcal{O}_D)} \leq \frac{\operatorname{covol} \left(\varphi(\mathcal{O}_D) \right)}{\operatorname{covol}(\mathcal{O}_D)} = \frac{\det(\varphi) \operatorname{covol}(\mathcal{O}_D)}{\operatorname{covol}(\mathcal{O}_D)}=\det(\varphi).
\]

Write $\displaystyle \varphi(\gamma_i)=\sum_{j=1}^r a_{ij} \gamma_j$ with $a_{ij} \in \mathbb{R}$ for the matrix representing $\varphi$ in the basis of the $\gamma_j$. Let $\lambda(\cdot,\cdot)$ be the bilinear form associated with $\deg$. Using the inequality $\deg(\varphi(e)) \leq  b(A/K)^2 \deg(e)$ we deduce
\[
\deg \left(\sum_j a_{ij} \gamma_j \right) = \deg(\varphi(\gamma_i)) \leq b(A/K)^2 \deg(\gamma_i)=b(A/K)^2 \quad \forall i=1,\ldots,r,
\]
so
\[
b(A/K)^2 \geq \lambda \left(\sum_j a_{ij} \gamma_j, \sum_k a_{ik} \gamma_k \right) = \sum_{j,k} a_{ij}a_{ik} \lambda(\gamma_j,\gamma_k) = \sum_j a_{ij}^2 \quad \forall i=1,\ldots,r:
\]
equivalently, the $L^2$-norm of each row of the matrix $(a_{ij})$ is bounded by $b(A/K)$. Hadamard's inequality then gives
\[
[\mathcal{O}_D : R] \leq \det(\varphi) \leq \prod_{i=1}^r \|a_i\|_{L^2} \leq b(A/K)^r,
\]
which is the desired estimate.
 \end{proof}

\medskip

\noindent\textbf{Acknowledgments.} I would like to thank my advisor, Nicolas Ratazzi, for his time and support.
This work was partially supported by the Fondation Mathématique Jacques Hadamard through the grant no ANR-10-CAMP-0151-02 in the ``Programme des Investissements d'Avenir".

\bibliography{Bibliography}{}
\bibliographystyle{alpha}

\end{document}